\input ./style/arxiv-general.cfg
\documentclass[MSNbibl,number,citesort,seceqn,dvips]{arxbj}
\makeatletter
   \@ifpackageloaded{graphicx}{}{\usepackage{graphicx}}
\makeatother
\usepackage{mathrsfs}
\volume{22}
\issue{4}
\pubyear{2016}
\firstpage{2209}
\lastpage{2236}
\doi{10.3150/15-BEJ726}% Updated by VTEXPTS2LaTeX.exe, 23.06.2015 10:20
\docsubty{FLA}

\makeatletter
\def\mid{\vert}
\newcommand{\N}{\mathbb N}
\newcommand{\Z}{\mathbb Z}

\newcommand{\R}{\mathbb R}
\newcommand{\E}{\mathbb E}
\renewcommand{\P}{\mathbb P}

\newcommand{\F}{\mathcal{F}}
\newcommand{\G}{\mathcal{G}}

\newcommand{\PP}{\mathscr{P}}

\newtheorem{thmm}{Theorem}[section]
\newtheorem{lem}[thmm]{Lemma}
\newtheorem{prop}[thmm]{Proposition}
\newtheorem{cor}[thmm]{Corollary}
\newremark{exmp}{Example}
\makeatother

\begin{document}
\begin{frontmatter}

\title{The genealogy of a solvable population model under selection
with dynamics related to directed polymers}
\runtitle{The genealogy of a solvable population model under selection}

\begin{aug}
%%%% inicialai - be tarpu
% Corresponding author: Aser Cortines - cortines@math.univ-paris-diderot.fr% Updated by VTEXPTS2LaTeX.exe, 23.06.2015 13:22
%cortines@math.univ-paris-diderot.fr% Updated by VTEXPTS2LaTeX.exe,
%23.06.2015 10:20
\author{\inits{A.}\fnms{Aser}~\snm{Cortines}\corref{}\ead[label=e1]{cortines@math.univ-paris-diderot.fr}}
\address{Universit\'e Paris Diderot, Math\'ematiques, case 7012, F-75
205 Paris Cedex 13, France.\\ \printead{e1}}
\runauthor{A. Cortines}
%\author[A]{\inits{}\fnms{}~\snm{}\corref{}\thanksref{A}
%\ead[label=e1]{}}%,
%\author[]{\fnms{}~\snm{}\thanksref{}\ead[label=]{}}
% \and
%\author[]{\inits{}\fnms{}~\snm{}\thanksref{}\ead[label=]{}}
%%\runauthor{} %% auto
%\dedicated{}
%\address[A]{. \printead{e1}}
%\address[]{. \printead{}}
\end{aug}

% HISTORY:
%
\received{\smonth{6} \syear{2014}}% Updated by VTEXPTS2LaTeX.exe,
%23.06.2015 10:20
%
\revised{\smonth{3} \syear{2015}}% Updated by VTEXPTS2LaTeX.exe,
%23.06.2015 10:20

% ABSTRACT
%
\begin{abstract}
We consider a stochastic model describing a constant size $N$
population that may be seen as a directed polymer in random medium with
$N$ sites in the transverse direction. The population dynamics is
governed by a noisy traveling wave equation describing the evolution of
the individual fitnesses. We show that under suitable conditions the
generations are independent and the model is characterized by an
extended Wright--Fisher model, in which the individual $i$ has a random
fitness $\eta_i$ and the joint distribution of offspring $ (\nu_1,
\ldots, \nu_N  )$ is given by a multinomial law with $N$ trials and
probability outcomes $\eta_i$'s. We then show that the average
coalescence times scales like $\log N$ and that the limit genealogical
trees are governed by the Bolthausen--Sznitman coalescent, which
validates the predictions by Brunet, Derrida, Mueller and Munier for
this class of models. We also study the extended Wright--Fisher model,
and show that, under certain conditions on $\eta_i$, the limit may be
Kingman's coalescent, a coalescent with multiple collisions, or a
coalescent with simultaneous multiple collisions.
\end{abstract}

% KEYWORDS
% visi is mazosios raides ir pagal abecele
%
\end{frontmatter}

%s1 #&#
\section{Introduction}\label{sec_intro_colescent}

An important question in the study of populations evolution is to
understand the effect of selection and mutation on its genealogy. For a
given population, we would like to know how individuals are related and
how many generations do we have to go back in time in order to find a
common ancestor. Kingman \cite{Kingman1982} was one of the first to
give a mathematical formulation for this problem and study the
ancestral history of a population. He showed that in the absence of
selection (neutral models) the populations genealogical structure
satisfies universal features; \textit{see also} \cite
{Mohle1999,Mohle2000,Mohle2001}.

In this paper, we focus on the evolution of a fixed size population
model with $N$ individuals subjected to the effects of mutation and
selection. We assign to each individual a real number, which represents
the fitness of this individual. This fitness is transmitted to the
offspring, up to variations due to mutations. Individuals with large
fitness spawn a considerable fraction of the population, whereas the
children of low fitness individuals tend to be eliminated. Therefore,
these population models are sometimes referred to as ``rapidly
adapting.'' If we consider the evolution of the fitnesses along the real
axis, it is simply a stochastic model of front propagation. The
selection mechanism constrains the particles to stay together. Since
individuals with large fitness quickly overrun the whole population,
the front is essentially pulled by the leading edge. These models are
then related to noisy traveling wave equations of the
Fisher--Kolmogorov--Petrovsky--Piscounov (FKPP) type \cite
{Brunet2013,Brunet2006,Brunet2007}.

Recent results suggest that in rapidly adapting population models the
genealogical correlations between individuals have universal features.
It is conjectured \cite{Brunet2013,Brunet2006,Brunet2007} that the
genealogical trees of these populations converge to the
Bolthausen--Sznitman coalescent and that the average coalescence times
scales like the logarithmic of the populations size. The conjectures
contrast with classical results in neutral population models, such as
Wright--Fisher and Moran models. It is known that in neutral population
models the genealogical trees converge to those of a Kingman's
coalescent and that the average coalescence times scales like $N$, the
size of the population~\cite
{Kingman1982,Mohle1999,Mohle2000,Mohle2001}. In Section~\ref{sec.coalescent.processes}, we will give a general introduction and
present some relevant results about coalescent processes.

We now mention some models, for which the conjectures have been proved.
The ``exponential model'' \cite{Brunet2013,Brunet2006,Brunet2007} is an
example of constant size population dynamics, for which a complete
mathematical treatment is possible. Each individual $i$ in generation
$t$ carries a value $x_i(t)$, which represents the fitness. The
offspring of the individuals are generated by independent Poisson point
process of densities $e^{-y+x_i(t)} \,\mathrm{d}y$. One then selects the
$N$ right-most individuals to form the next generation $t+1$. The
authors show that, after rescaling time by a factor $\log N$, one
obtains the convergence to the Bolthausen--Sznitman coalescent.
Berestycki, Berestycki and Schweinsberg \cite{Berestycki2013}
consider a system of particles, performing branching Brownian motion
with negative drift and killed upon reaching zero. The authors choose
the appropriate drift such that the model is in the near-critical
regime and the initial population size $N$ is roughly preserved. They
show that the expected time to observe a merge is of order $(\log N)^3$
and that the genealogy of the particles is also governed by the
Bolthausen--Sznitman coalescent.

We also mention other related models, for which the genealogical trees
do not converge to those of a Kingman's coalescent. Schweinsberg \cite
{Schweinsberg2003} considers a model, in which the numbers of offspring
for the individuals are i.i.d. (Galton--Watson processes), but in each
generation only $N$ of the offspring are chosen at random for survival
(selection mechanism). He proves that depending on the tail
probabilities of the reproductive law, the limit may be Kingman's
coalescent, a coalescent with multiple collisions ($\Lambda
$-coalescent), or a coalescent with simultaneous multiple collisions
($\Xi$-coalescent). The authors in \cite{Huillet2013} study the
asymptotic of the extended Moran model as the total population size $N$
diverges, and show that the ancestral process of the population may be
approximated by a coalescent process with multiple collisions. Discrete
population models with unequal (skewed) fertilities, such as the skewed
Wright--Fisher model and the Kimura model, are not necessarily in the
domain of attraction of the Kingman's coalescent \cite{Huillet2011}.

In the present paper, we consider a population dynamics derived from
the following model of front propagation \cite{Brunet2004}. It consists
in a constant number $N$ of evolving particles on the real line
initially at positions $X_1(0), \ldots,X_N(0)$. Then, given the
positions $X_i(t)$ of the $N$ particles at time $t \in\N$, we define
the positions at time $t + 1$ by
%
%e1.1 #&#
\begin{equation}
\label{definition.X.derrida.brunet.coalescent} X_j (t + 1) : = \max_{1\leq i \leq N} \bigl
\{ X_i(t) + \xi_{ij} (t + 1) \bigr\},
\end{equation}
where $ \{ \xi_{ij} (s);  1 \leq i, j \leq N,   s \in\N \} $
are i.i.d. real random variables. The model can be seen as a directed
polymer in random medium at zero temperature. The lattice consists in
$L$ planes in the transversal direction. In every plane, there are $N$
points that are connected to all points of the previous plane and the
next one and for each edge $ij$, connecting the planes $t$ and $t+1$, a
random energy $-\xi_{ij}(t+1)$ is sampled from a common probability
distribution. At zero temperature, the directed polymer chooses the
path which minimizes its energy (the optimal path) and $-X_j(L)$ is
equal to minimal energy among all paths connecting the origin to the
$j$th point on the $L$th plane~\cite{Cook1990,Cortines2013}. The
optimal path starting at the same point but arriving at different
points gives rise to a tree structure. It is well known that population
dynamics in presence of selection may be related to directed polymers
in random medium at zero temperature and it is expected that they
belong to the same universality class \cite
{Brunet2013,Brunet2008,Brunet2006,Brunet2007}.

If the distribution of $\xi_{ij}(t+1)$ in (\ref
{definition.X.derrida.brunet.coalescent}) has no atoms, \textit{that is,}
for every $x \in\R$ the probability $\P(\xi_{ij}(t+1) = x) = 0$, then
for all $j$ the following equation has a.s. a unique solution $i$:
%
%e1.2 #&#
\begin{equation}
X_j(t+1) = X_{i}(t)+\xi_{i j}(t+1).
\end{equation}
In this sense, we may say that $X_j(t+1)$ is an offspring or a
descendant of $X_{i}(t)$ and denote by $\nu_{i}(t)$ the number of
descendants of $X_{i}(t)$ in generation $t+1$. The fitnesses of the
individuals are given by their positions $X_1(t), \ldots, X_N(t)$ and
conditionally on
\[
\F_t := \sigma\bigl\{ \xi_{ij}(s) \mbox{ and }
X_i(0); 0 \leq s \leq t, 1 \leq i,j \leq N \bigr\}, %
\]
the probability that $X_j(t+1)$ descends from $X_{i}(t)$ is given by
%
%e1.3 #&#
\begin{equation}
\label{equa:general.fitness.coef} \eta_{i} (t) := \P \bigl( \xi_{i j}(t+1)
+X_{i }(t) \geq\xi_{kj}(t+1) +X_k(t) ; \mbox{
for every } 1 \leq k \leq N | \F_t \bigr).
\end{equation}
Since $ \{ \xi_{ij} (t+1) ;   1 \leq i, j \leq N  \} $ are
independent, it is easy to see that, for $j_1, \ldots, j_m$ distinct
and $i_1, \ldots, i_m $ (not necessarily distinct),
\begin{eqnarray*}
&&\P \bigl( X_{j_k}(t+1) \mbox{ descends from } X_{i_k}(t),
\mbox{ for } 1 \leq k \leq m | \F_t \bigr)
\\
%&= \P\big( X_{j_k}(t+1) = X_{i_k}(t)+\xi_{i_k j_k}(t+1), \mbox{ for }
%1 \leq k \leq m \big| \F_t \big) \\
&&\quad = \eta_{i_1} (t) \eta_{i_2} (t) \cdots
\eta_{i_m} (t).
\end{eqnarray*}
If $i_k = i_l$, the individuals $j_k$ and $j_l$ have a common ancestor
in generation $t$. As a consequence, given $\F_t$ the offspring vector
$ \nu(t) :=  (\nu_1(t), \ldots,\nu_N (t)  )$ is distributed
according to a $N$-class multinomial with $N$ trials and probabilities
outcomes $\eta(t) :=  (\eta_1(t), \ldots,\eta_N (t)  )$.

We analyse the genealogical tree of the population by observing the
ancestral partition process, that is, we sample without replacement
$n \ll N$ individuals from a given generation $T$, say $e_1, \ldots,
e_n$ and for $0 \leq t \leq T$ we consider $\Pi_{t}^{N,n}$ the random
partition of $[n]:= \{1,\ldots,n\}$ such that $i$ and $j$ belong to the
same equivalent class if and only if $e_i$ and $e_j$ share the same
ancestor at time $T-t$. It is very important to realize that the
direction of time for the ancestral process is the opposite of the
direction of time for the ``natural'' evolution of the population.

If $\xi_{ij}$ in (\ref{definition.X.derrida.brunet.coalescent}) is
Gumbel $G(\rho, \beta)$-distributed, \textit{that is,}
\[
\P(\xi_{ij} \leq x) = \exp \bigl(-e^{-\beta(x-\rho)} \bigr),\qquad  x \in\R,
\]
the microscopic dynamics can be solved allowing precise calculations;
see also \cite{Brunet2004} where Brunet and Derrida use a similar
technique to compute the exact asymptotic for the speed of the front.
In this case, \emph{see Proposition~\textup{\ref
{prop.indep.offspring.generations.gumbel}} in Section~\textup{\ref{sec.derrida}}}, the positions of the particles in generation $t+1$ can
be obtained by a $\F_t$-measurable function $\Phi (X(t) )$ (that
may be interpreted as the front position at time $t$) and a $\F
_t$-independent family of i.i.d. random variables $ ( \mathcal{E}_i
(t+1)  )_{1 \leq i \leq N}$
%
%e1.4 #&#
\begin{equation}
X_i(t+1) = \rho + \Phi \bigl(X(t) \bigr) - \beta^{-1}
\log\mathcal{E}_i (t+1).
\end{equation}
Hence, one only needs the information $ \Phi (X(t) )$ from $\F_t$
to generate the particle position $X_i(t+1)$. In this case, one obtains
the following weak limit for the ancestral partition process.

%th1.1 #&#
\begin{thmm}\label{thmm.genealogy.Gumbel.case} Assume that $\xi_{ij}$
in (\ref{definition.X.derrida.brunet.coalescent}) are Gumbel $G(\rho,
\beta)$-distributed and that the initial position of particles $
(X_1(0), \ldots, X_N(0) )$ are distributed according to a
probability distribution $\mu$ on $\R^{N}$. Choose $n$ particles $e_1,
\ldots, e_n$ uniformly at random from the $N$ particles in generation
$ \lfloor T (\log N)  \rfloor$. Let\vspace*{1pt} $ (\Pi^{N,n}_{\lfloor t
(\log N) \rfloor} ; t \in [0,T  [  )$ be the random
partition of $[n]$ such that $i$ and $j$ are in the same block if and
only if $e_i$ and $e_j$ have the same ancestor in generation $
\lfloor(T-t) (\log N)  \rfloor$.

Then the processes $ (\Pi^{N,n}_{\lfloor t (\log N) \rfloor} ; t
\in [0, T  [  )$ converge weakly as $N \to\infty$ to a
continuous time process $ (\Pi^{\infty,n}_t ; t \in [0,T  [
 )$ that has the same law as the restriction to $[n]$ of the
Bolthausen--Sznitman coalescent (up to time $T^-$).
\end{thmm}

The exponential model \cite{Brunet2013,Brunet2006} and the population
dynamics in (\ref{definition.X.derrida.brunet.coalescent}) share the
common property that the only information one needs from generation $t$
in order to obtain generation $t+1$ is contained on a single function
of the particles positions at time $t$. Using this property and
shifting the particles positions appropriately, one can prove, for
example, the independence between generations. Yet, it is important to
point out that the techniques used to prove independence and the
population models are different. In the exponential model, each
individual has infinitely many offspring, but only the $N$ right-most
are selected to form the next generation. On the other hand, in (\ref
{definition.X.derrida.brunet.coalescent}) each individual has only $N$
offspring and the selection mechanism is of a different nature. Indeed,
we may label the offspring of $X_i(t)$ according to the $\xi
_{ij}(t+1)$'s, so the child labeled $j \in\{1, \ldots, N\}$ is at
position $X_i(t) + \xi_{ij} (t+1)$. The selection is then made among
individuals having the same label: $X_j(t+1) = \max_{1\leq i \leq N} \{
X_i(t) +\xi_{ij} (t+1) \}$, and in generation $t+1$ we keep the
right-most individual of each label $j$ and not the $N$ right-most
individuals, as in the exponential model. Hence, Theorem~\ref
{thmm.genealogy.Gumbel.case} provides an other example of population
dynamics in the presence of selection (or directed polymer in random
medium) that validates the conjectures in \cite
{Brunet2013,Brunet2008,Brunet2006,Brunet2007}.

The population dynamics obtained from (\ref
{definition.X.derrida.brunet.coalescent}) can be described as follows.
The individuals in generation $t$ have a (random) genetic fitness $\eta
_i(t)$, that determines their average reproductive success. The total
genetic fitness is a.s. constant and equal to one, $\sum\eta_i(t) =
1$, then given $\eta(t)$ the offspring vectors $ (\nu_1(t),\ldots,
\nu_N(t) )$ are distributed according to a $N$-class multinomial
with $N$ trials and probabilities outcomes $\eta_i(t)$'s. If we assume
that the offspring vectors $ (\nu(t)  )_{t \in\N}$ are
identically distributed and independent from generation to generation,
then we obtain a ``toy model,'' in which generations are not correlated.
In this paper, we also study the ancestral history of this population.
We make two additional assumptions on the fitness $\eta(t)$. First, we
assume that each $\eta_i(t)$ is of the form
%
%e1.5 #&#
\begin{equation}
\label{equa.def.eta} \eta_i(t) = Y_i(t) \Big/ \sum
_{j=1}^N Y_j (t),
\end{equation}
where $Y_j(t)$ are i.i.d. positive random variables. Secondly, for some
of our results, we assume that the tail distribution of $Y_i(t)$ satisfies
%
%e1.6 #&#
\begin{equation}
\label{equa.cond.Y.alpha} \lim_{y \to\infty}\P \bigl(Y_i(t) \geq y
\bigr) / y^{-\alpha} = C,
\end{equation}
where $\alpha$ and $C$ are positive constants. To simplify the
notation, the time parameter $t$ is often omitted. Moreover, $\eta
_i(t)$ in (\ref{equa.def.eta}) does not change if we replace $Y_j(t)$
by $Y_j(t)C^{ -1/\alpha}$, for this reason we may always assume that
$C=1$. Then we show that the ancestral processes converge weakly and
that the limit distribution depends on $\alpha$. Our result resembles
Theorem~4 in \cite{Schweinsberg2003}, where Schweinsberg studies
coalescent processes obtained from supercritical Galton--Watson processes.

%th1.2 #&#
\begin{thmm}\label{thmm.general.criteria.independent.case} Consider the
dynamics of a constant size $N$ population with infinitely many
generations backward in time defined by the vectors $\nu(t) =  (\nu
_1(t),\ldots, \nu_N(t) ),   t \in\Z$ of family sizes and denote by
$\Pi^{N,n}_t$ the ancestral partition process. Suppose that the family
sizes $\nu(t)$ are i.i.d. copies of $\nu$ a doubly stochastic
multinomial random variable with $N$ trials and probability outcomes
$\eta=  (\eta_1,\ldots, \eta_N ) $:
\[
\P \bigl( \nu= (i_1, \ldots, i_N ) | \eta \bigr) =
\frac
{N!}{i_1! \cdots i_N!} \eta_1^{i_1} \cdots \eta_N^{i_N},
\]
where $i_1,\ldots, i_N \in\N$ and $i_1+ \cdots+ i_N =N$. Suppose also
that $\eta_i$ is of the form (\ref{equa.def.eta}) with i.i.d. $Y_i$'s.
Then the following holds.
\begin{longlist}[(a)] \label
{thmm.general.criteria.independent.case;square.integrable}
\item[(a)] If $\E[Y_1^2] < \infty$ (in particular, if (\ref
{equa.cond.Y.alpha}) holds and $\alpha>2$), then the processes $ (\Pi
^{N,n}_{\lfloor t / c_N \rfloor} ; t \geq0  )$ converge weakly as
$N \to\infty$ to the Kingman's n-coalescent. The scaling factor $c_N$
is asymptotically equivalent to $N$, precisely
\[
\lim_{N\to\infty} N c_N = \frac{\E[Y_i^2]}{\E[Y_i]^2}. %
\]
\item[(b)] If the $Y_i$'s satisfy (\ref{equa.cond.Y.alpha}) with $\alpha
=2$, then the processes $ (\Pi^{N,n}_{\lfloor t /c_N \rfloor} ; t
\geq0  )$ converge in the Skorokhod sense as $N \to\infty$ to
the Kingman's n-coalescent. The scaling factor $c_N$ is asymptotically
equivalent to $N/\log N $
\[
\lim_{N\to\infty} \frac{N c_N }{\log N} = \frac{2}{\E[Y_i]^2}. %
\]
\item[(c)] When (\ref{equa.cond.Y.alpha}) holds with $1 \leq\alpha< 2$,
then the processes $ (\Pi^{N,n}_{\lfloor t /c_N \rfloor} ; t \geq
0 )$ converge in the Skorokhod sense as $N \to\infty$ to a
continuous-time process $ (\Pi^{\infty,n}_t ; t \geq0  )$
that has the same law as the restriction to $[n]$ of the $\Lambda
$-coalescent, where $\Lambda$ is the probability measure associated
with the $\operatorname{Beta}(2 - \alpha; \alpha)$ distribution. The transition rates
are given by
%
%e1.7 #&#
\begin{equation}
\lambda_{b; k} = \frac{ B(k - \alpha; b - k + \alpha)}{B(2 - \alpha;
\alpha)},
\end{equation}
where $B(c,d)= \Gamma(c) \Gamma(d)/ \Gamma(c+d)$ is the beta function.
The scaling factor $c_N$ satisfies
\begin{eqnarray*}
\lim_{N\to\infty} N^{\alpha-1} c_N &=&
\frac{ \alpha\Gamma(\alpha)
\Gamma(2-\alpha)}{\E[Y_i]^{\alpha}}\qquad \mbox{if } 1<\alpha<2,
\\
\lim_{N\to\infty} c_N \log N &=& 1\qquad \mbox{if } \alpha=1.
\end{eqnarray*}
\item[(d)]%\label{item.poisson.dirichlet}
When (\ref{equa.cond.Y.alpha})
holds with $0 < \alpha< 1$, then the processes $ (\Pi^{N,n}_{ t }
; t \in\N )$ converge as $N \to\infty$ to a discrete-time
Markov chain $ (\Pi^{\infty,n}_t ; t \in\N )$ that has the
same law as the restriction to $[n]$ of a discrete-time $\Xi_\alpha
$-coalescent. The transition probabilities are given by
%
%e1.8 #&#
\begin{equation}
p_{b; b_1;\ldots; b_a ; s} = \frac{\alpha^{a+s-1}(a + s - 1)!}{(b -
1)!} \cdot\prod_{i=1}^{a}
\frac{\Gamma(b_i - \alpha)}{\Gamma(1-\alpha)}.
\end{equation}
\end{longlist}
\end{thmm}

Despite the similarities between Theorem~\ref
{thmm.general.criteria.independent.case} and Theorem~4 in \cite
{Schweinsberg2003}, we consider a population dynamics that is different
from the one studied by Schweinsberg. In \cite{Schweinsberg2003}, each
individual gives birth to $\zeta_i(t) \in\N$ children, but only $N$
among the $\zeta_1(t)+ \cdots+\zeta_N(t)$ survive. The survivors are
chosen uniformly without replacement and $\nu_i(t)$ is the number of
descendants that remain after the selection step. The distribution of
$ (\nu_1(t), \ldots, \nu_N(t) )$ is then characterized by an urn
model. Indeed, if $\zeta_i(t),  1\leq i \leq N$ is the number of balls
in the urn which are labeled $i$, so $\nu_i$ is the number of $i$-balls
sampled after $N$ draws without replacement. If we suppose that the
$Y_i$ in Theorem~\ref{thmm.general.criteria.independent.case} are
integer valued, we may also compare the population dynamics with an urn
model. In this case, though, $\nu_i$ is the number of $i$-balls sampled
after $N$ draws with replacement. Then $ (\nu_1(t), \ldots, \nu
_N(t) )$ is distributed according to a multinomial with $N$ trials
and probability outcomes $Y_i(t)  /  (Y_1(t)+ \cdots+Y_N(t)
)$ and we are under the hypotheses of Theorem~\ref
{thmm.general.criteria.independent.case}. On the other hand, the $Y_i$
are not necessarily integer valued and may be distributed according to
any distribution satisfying (\ref{equa.cond.Y.alpha}). In fact, as the
reader will see in the proof of Theorem~\ref
{thmm.genealogy.Gumbel.case}, a~relevant case is when the $Y_i$ are
distributed according to the inverse of an exponential distribution.
Whereas in \cite{Schweinsberg2003}, $\zeta_i(t)$ must be an integer
valued random variable, since it represents the number of offspring of
the $i$th individual in generation $t$.

The paper is organized as follows: in Section~\ref{sec.coalescent.processes}, we recall some necessary definitions and
results about coalescent processes. Then, in Section~\ref{sec.derrida},
we study the case where the disorder $\xi_{ij}$ is Gumbel distributed
and we obtain Theorem~\ref{thmm.genealogy.Gumbel.case} as an
application of Theorem~\ref{thmm.general.criteria.independent.case},
that will be proved later in Section~\ref{sec.proof.theorem}. At the
end of the paper, we include two Appendices, in which we prove some
technical results.

%s2 #&#
\section{Coalescent processes} \label{sec.coalescent.processes}

Let $\PP_n$ be the finite set of all partitions of $[n]$ and $\PP_\infty
$ the set of partitions of $\N^*$. For $\pi, \pi' \in\PP_n$ we say
that $\pi'$ is a refinement of $\pi$ if every equivalent class of $\pi$
is either a union of several equivalence classes of $\pi'$ or coincides
with an equivalence class of $\pi'$, we denote it by $\pi' \subset\pi$.

We call a $\PP_n$-valued process $ (\Pi^{n}_t; t\geq0  )$ a
$n$-coalescent if it has right-continuous step function paths and if
$\Pi^n_s$ is a refinement of $\Pi^n_t$, whenever $s \leq t$. We call a
$\PP_\infty$-valued process $ (\Pi_t; t\geq0  )$ a coalescent
if it has c\`adl\`ag paths and if $\Pi_s$ is a refinement of $\Pi_t$
for all $s < t$. In this paper, we use the notation $\Pi^{N,\cdot}$ to
denote the ancestral partition of a constant size population with $N$
individuals, while the notation $\Pi^{\infty,\cdot}$, or simply $\Pi$,
stands for a coalescent process.

We denote by $\mathcal{D} ([0,\infty);\PP_n )$ the space of c\`
adl\`ag functions on $[0,\infty)$ taking values in $\PP_n$, obviously
$ (\Pi^{n}_t; t\geq0  ) \in\mathcal{D} ([0,\infty);\PP_n
) $. Since $\PP_n$ endowed with the discrete metric is a separable
complete metric space, the space $\mathcal{D} ([0,\infty);\PP_n
)$ is also separable and complete in the Skorokhod distance. We say
that a process converges in the Skorokhod sense if the distribution of
the process converges weakly in $\mathcal{D} ([0,\infty);\PP_n )$
equipped with this metric.

%s2.1 #&#
\subsection{\texorpdfstring{$\Lambda$}{Lambda}-coalescent}

In \cite{Pitman1999}, Pitman studied the so-called $\Lambda
$-coalescent. It consists in ``coalescents with multiple collisions''
that are continuous time Markov chains taking value in $\PP_\infty$.
$\Lambda$-coalescents have the property that the rate at which blocks
are merging does not depend on the size of the blocks nor on the
integers that are in the blocks, moreover simultaneous collisions do
not happen. Let $\lambda_{b,k}$ be the rate that $k$ blocks merge into
a single one when there are $b$ blocks in total. The array $ (\lambda
_{b,k} )_{2 \leq k \leq b }$ determines the distribution of $\Pi
^n$'s and, consequently, the distribution of $\Pi$. As Pitman shows in
\cite{Pitman1999}, there exists a coalescent process with transition
rates $\lambda_{b,k}$ if and only if the consistency condition
\[
\lambda_{b,k} = \lambda_{b+1,k}+ \lambda_{b+1,k+1}
\]
holds. In this case, there exists a non-negative and finite measure on
the Borel subsets of $[0,1]$ such that
\[
\lambda_{b,k} = \int_{[0,1]} u^{k-2}
(1-u)^{b-k} \Lambda(\mathrm{d} u). %
\]
The process is then called the $\Lambda$-coalescent. When $\Lambda$ is
a unit mass at zero, we obtain the Kingman's coalescent. Another
notorious case is when $\Lambda$ is the uniform distribution on
$[0,1]$; this process was studied by Bolthausen and Sznitman in \cite
{Bolthausen1998} and is named after the authors.

One can further generalize these processes and obtain $\PP_\infty
$-Markov processes that may undergo ``simultaneous multiple
collisions,'' the $\Xi$-coalescent, \textit{see M\"ohle and Sagitov \cite
{Mohle2001} and Schweinsberg \cite{Schweinsberg2000}}. Let $b, b_1,
\ldots, b_a, s$ be non-negative integers such that $b_1\geq\cdots\geq
b_a \geq2$ and $b= s + \sum b_i$. Then, $\Xi$-coalescent are $\PP
_\infty$-Markov processes characterized by the rates $\lambda_{b; b_1,
\ldots, b_a; s } $ at which $b$ blocks merge into $a+s$ blocks, with
$s$ blocks that remain unchanged and $a$ blocks that are obtained by
the union of $b_1, \ldots, b_a$ blocks before the merging. As M\"ohle
and Sagitov observe in Lemma~3.3 of \cite{Mohle2001} (see also
Schweinsberg \cite{Schweinsberg2000}) the transition rates satisfy the
following recursion:
%
%e2.1 #&#
\begin{equation}
\label{equa.recursion.lambda} \lambda_{b; b_1, \ldots, b_a ; s +1} = \lambda_{b; b_1, \ldots, b_a ; s
} - \sum
_{j=1}^{a} \lambda_{b +1 ; b_1, \ldots, b_j +1 , \ldots, b_a ;
s }- s
\lambda_{b+1; b_1, \ldots, b_a, 2 ; s-1 }.
\end{equation}
Hence, the distribution of a $\Xi$-coalescent is completely determined
by the rates $\lambda_{b; b_1, \ldots, b_a}$.

%s2.2 #&#
\subsection{Weak convergence of ancestral processes}\label
{subsec.weak.conv.ancestral}

It is well known that coalescent processes may be obtained as the weak
limit of ancestral processes \cite{Mohle1999,Mohle2000,Mohle2001}. M\"
olhe and Sagitov study a wide class of constant size population models,
which have ``been living forever'' (so we may trace back the individuals
genealogical tree indefinitely). They obtain general conditions under
which the ancestral processes $\Pi_{t}^{N,\cdot}$ converge in the
Skorokhod sense to a coalescent process. As usual denote by $\nu_i(t)$
the number of children of the $i$th individual in generation $t$
\[
\nu_1(t) + \nu_2(t) + \cdots+ \nu_N(t) =
N,\qquad t \in\Z. %
\]
They assume that generations do not overlap and that the family sizes
in different generations are i.i.d. Generally, it is also assumed that
individuals in a given generation have the same propensity to reproduce:
\begin{longlist}[(ii)]
\item[(i)]\label{cond:reproduction.vect.I}The offspring vectors $\nu(t) ,
  t \in\Z$ are i.i.d. copies of $\nu$.
\item[(ii)]%\label{cond:reproduction.vect.II}
The offspring vector $ (\nu
_1,\ldots, \nu_N  )$ is $N$-exchangeable.
\end{longlist}
The first assumption is necessary since it ensures the Markov property
of the ancestral partition process. Under the above assumptions, it is
easy to compute the transition probability of $\Pi^{N,n}$. Let $\pi'
\subset\pi$ be two partitions of $\PP_n$ and denote by $a$ and $b$ the
number of equivalent classes of $\pi$ and $\pi'$, respectively. Then
$b$ may be decomposed as follows: $b=b_1+ \cdots+ b_a$, where $b_i$'s
are ordered positive integers denoting the number of equivalent classes
of $\pi'$ that we have to merge in order to obtain one equivalent class
of $\pi$. By a combinatorial ``putting balls into boxes'' argument, we
obtain that the transition probability from $\pi'$ to $\pi$ is
%
%e2.2 #&#
\begin{eqnarray}
\label{equa:transition.prob.general.model} p_N \bigl(\pi', \pi\bigr) &= &\P \bigl(
\Pi_{t+1}^{N,n} = \pi\mid\Pi_{t}^{N,n}=
\pi ' \bigr)
\nonumber
\\[-8pt]
\\[-8pt]
\nonumber
& = &\frac{1}{(N)_b} \mathop{\sum_{i_1, \ldots, i_a =1 }}_{\mathrm{all\ distinct}}^{N}
\E \bigl[ (\nu_{i_1})_{b_1} \cdots(\nu_{i_a})_{b_a}
\bigr] ,
\end{eqnarray}
where $(N)_b:= N (N-1)\cdots(N-b+1)$. If the offspring vector is
$N$-exchangeable we may further simplify (\ref
{equa:transition.prob.general.model}) obtaining
\[
p_N \bigl(\pi', \pi\bigr) = \frac{(N)_a}{(N)_b} \E
\bigl[ (\nu_{1})_{b_1} \cdots (\nu_{a})_{b_a}
\bigr]. %
\]
We now state M\"olhe and Sagitov result, we keep their notation and let
$c_N$ be the probability that two individuals, chosen randomly without
replacement from some generation, have a common ancestor one generation
backward in time (it is the same $c_N$ appearing in the statement of
Theorem~\ref{thmm.general.criteria.independent.case}).
%
%e2.3 #&#
\begin{equation}
\label{equa:defi.cN} c_N: = \frac{1}{N(N-1)} \sum
_{i}^{N} \E \bigl[\nu_i(t) \bigl(
\nu_i(t)-1\bigr) \bigr] = \frac{1}{(N-1)} \E \bigl[
\nu_1(t) \bigl(\nu_1(t)-1\bigr) \bigr].
\end{equation}

%th2.1 #&#
\begin{thmm}[(M\"olhe and Sagitov \cite{Mohle2001})] \label
{thmm:mohle.sagitov} Suppose that for all $a \geq1$ and $b_1 \geq
\cdots\geq b_a \geq2$, the limits
%
%e2.4 #&#
\begin{equation}
\label{equa.thmm:mohle.sagitov} \lim_{N \to\infty} \frac{ \E [ (\nu_{1})_{b_1} \cdots(\nu
_{a})_{b_a}  ] }{N^{b_1+\cdots+b_a-a} c_N}
\end{equation}
exist, and let $b:=b_1+\cdots+b_a$. If
\[
\lim_{N\to\infty} c_N = 0, %
\]
then the time-rescaled ancestral processes $  (\Pi^{N,n}_{\lfloor
t/c_N \rfloor},   t \geq0  ) $ converge weakly as $N \to\infty
$ to a process $ (\Pi^{\infty,n}_t,   t \geq0 ) $ that has
the same law as the restriction to $[n]$ of a $\Xi$-coalescent.
Furthermore, the transition rates $\lambda_{b; b_1, \ldots, b_a }$,
that characterize the distribution of $\Pi^{\infty,n}_t$, are equal to
the limits in (\ref{equa.thmm:mohle.sagitov}). On the other hand, if
\[
\lim_{N\to\infty} c_N = c >0, %
\]
then the processes $ (\Pi^{N,n}_{ t },   t \in\N )$ converge
weakly as $N \to\infty$ to a process $ (\Pi^{\infty,n}_t,   t \in
\N ) $, which has the same law as the restriction to $[n]$ of a
discrete-time $\Xi$-coalescent. The transition probabilities $p_{b;
b_1,\ldots,b_a }$ satisfy
%
%e2.5 #&#
\begin{equation}
p_{b; b_1,\ldots,b_a } = \lim_{N \to\infty} \frac{ \E [ (\nu
_{1})_{b_1} \cdots(\nu_{a})_{b_a}  ] }{N^{b_1+\cdots+b_a-a}}.
\end{equation}
\end{thmm}

The existence of the limits in (\ref{equa.thmm:mohle.sagitov}) implies
that the finite-dimensional distributions of $\Pi^{N,n}_{\lfloor t /c_N
\rfloor}$ converge to those of the coalescent $\Pi^n_t$, as proved in
\cite{Mohle2001}. The authors in \cite{Mohle1999,Mohle2001} prove that
when $c_N \to0$ the sequence of processes $\Pi^{N,n}_{\lfloor t /c_N
\rfloor}$ is tight, which implies the weak convergence in the Skorokhod sense.

%s3 #&#
\section{Relation with Brunet and Derrida's model}\label{sec.derrida}

In this section, we will assume that Theorem~\ref
{thmm.general.criteria.independent.case} holds and we show that when
the $\xi_{ij}$'s are Gumbel distributed, then the family sizes $\nu(t)$
of the model (\ref{definition.X.derrida.brunet.coalescent}) are i.i.d.
and the distribution satisfies the hypotheses of Theorem~\ref
{thmm.general.criteria.independent.case} with $\alpha=1$, which implies
Theorem~\ref{thmm.genealogy.Gumbel.case}. We bring to the reader's
attention two important details.

The first one is that the time restriction in the statement of Theorem~\ref{thmm.genealogy.Gumbel.case} is a necessary condition. One
immediate reason is that the ancestral process is not even defined for
\mbox{$t > T$}. A more subtle reason is that the partition $\Pi^{N,n}_{
\lfloor T (\log N) \rfloor}$ depends on the initial distribution
$X_1(0), \ldots, X_N(0)$. This dependence can be easily illustrated by
the following example. One chooses an initial position of points:
$X_1(0), \ldots, X_N(0)$, for which $X_1(0) \gg X_i(0)$. Then, with an
overwhelming probability, every individual in generation one descends
from $X_1(0)$ and
\[
\Pi^{N,n}_{ \lfloor T (\log N) \rfloor} = \bigl\{(1,\ldots,n)\bigr\}, %
\]
in particular, as $N\to\infty$ the partition $\Pi^{N,n}_{ \lfloor T
(\log N) \rfloor}$ does not converge in distribution to the
$n$-Bolthausen--Sznitman coalescent at time $T$.

Secondly, we emphasize that, in the general case, the offspring vectors
$\nu(t)$ obtained from~(\ref{definition.X.derrida.brunet.coalescent})
may not be independent from generation to generation. We refer to \cite
{Cortines2013} to provide a picture of a situation, in which the
positions of the particles are highly related to the positions of their
ancestors. It is considered the case, in which the distribution of $\xi
_{ij}$ depends on $N$
\[
\P ( \xi_{ij} = 0 ) = 1- \P ( \xi_{ij} = -1 ) =
1/N^{1+r}. %
\]
In this model, the number of leaders $\sharp \{i; X_i(t)= \max\{
X_j(t)\}  \}$ in generation $t$ has a strong correlation with the
number of leaders in generation $t-1$. Therefore, the fitness vectors
$ (\eta(t) )_{t \in\N}$ between successive generations are
correlated, and hence the offspring vectors $\nu(t)$ are not
independent (in particular (i) in page
\pageref{cond:reproduction.vect.I} does not hold).

Before proving Theorem~\ref{thmm.genealogy.Gumbel.case}, let us present
some preliminary results and explain why the Gumbel case is particular.
In \cite{Comets2013}, it is shown that the particles remain grouped as
$t$ increases and that the position of the front at time $t$ may be
described by any numerical function $\Phi: \R^N \rightarrow\R$ that
is increasing for the partial order on $\R^N$ and that commutes to
space translations by constant vectors
%
%e3.1 #&#
\begin{equation}
\Phi(x + r \mathbf{1}) = r + \Phi(x),
\end{equation}
where $\mathbf{1}$ is the vector $(1, 1, \ldots, 1) \in\R^N$. For a
given function $\Phi$, we denote by $x^0$ the vector $x \in\R^N$
shifted by $\Phi(x)$.
\[
x^0 = x - \Phi(x). %
\]
The authors also prove that there exists a non-random constant $v_N$
(not depending on $\Phi(\cdot)$) called speed of the front such that
\[
\lim_{t\to\infty} \frac{\Phi (X(t) )}{t} = v_N\qquad \mbox{a.s.}
\]
It is then clear that there is no invariant measure for $X(t):= (
X_1(t), \ldots, X_N(t)  )$. On the other hand, if we consider the
shifted process $X^0 (t):= X(t)-\Phi (X(t) )$, then there exists
a unique invariant measure (depending on $\Phi(\cdot)$) for it. In the
Gumbel case, an appropriate measure of the front location is
%
%e3.2 #&#
\begin{equation}
\label{equa:front.position.gumbel.brunet.derrida} \Phi(x) = \beta^{-1} \log\sum
_{i=1}^{N} \exp(\beta x_i).
\end{equation}
In the proof of Proposition~\ref
{prop.indep.offspring.generations.gumbel}, we show that if the $\xi
_{ij}$ are Gumbel $G(\rho,\beta)$-distributed, then $\Phi ( X(t)
)$ has all information needed to construct the next generation. The
technique that we will present has been used in \cite{Brunet2004} to
calculate the velocity and diffusion constant of the $N$-particles
system. In \cite{Comets2013}, the authors use the same argument to
calculate explicitly the invariant measure for the process $X^0(t)$. It
has the law of a shifted vector $V^0:=V-\Phi(V)$ of a vector $V$
obtained from a $N$-sample from a Gumbel $G(0,\beta)$. Summing up, when
the disorder is Gumbel distributed the model is completely soluble,
allowing exact computations.

%pr3.1 #&#
\begin{prop}\label{prop.indep.offspring.generations.gumbel} Assume that
$\xi_{ij}$ in (\ref{definition.X.derrida.brunet.coalescent}) are Gumbel
$G(\rho,\beta)$-distributed and denote by $\nu_i(t)$ the number of
descendants of $X_i(t)$ in generation $t+1$.

Then, for every starting configuration $\mu$ the family sizes $\nu(t)=
 ( \nu_1(t), \ldots, \nu_N(t) ),   t\geq1$ are i.i.d. copies of
$\nu$ a doubly stochastic multinomial random variable with $N$ trials
and probability outcomes $\eta_i$ given by
%
%e3.3 #&#
\begin{equation}
\label{equa:defi.coef.fitness.gumbel.final} \eta_i = \mathcal{E}^{-1}_i
\Big/ \Biggl( \sum_{k=1}^{N} \mathcal
{E}^{-1}_k \Biggr),
\end{equation}
where $ \{\mathcal{E}_i;   1 \leq i \leq N  \}$ are independent
and exponentially distributed with parameter $1$. If $\mu$ has the law
of a shifted vector $V^0:=V-\Phi(V)$ of a vector $V$ obtained from a
$N$-sample from a Gumbel $G(0,\beta)$, then we may take $t \geq0$.
\end{prop}

\begin{pf} Let $\Phi(x)$ be given by (\ref
{equa:front.position.gumbel.brunet.derrida}), then $\Phi(x)$ has all
information one needs to construct the next generation and the process
shifted by $\Phi$: $X^0_j(t) = X_j(t)- \Phi ( X(t)  )$, are
independent from generation to generation. Indeed, for $t\geq1$ we may
write $X_j(t)$ as follows \textit{(see \cite{Comets2013} (Theorem~3.1) and
\cite{Brunet2004}):}
%
%e3.4 #&#
\begin{equation}
\label{equa:decomposition,gumbel,case} X_j(t) = \rho + \Phi \bigl(X(t-1) \bigr) -
\beta^{-1} \log\mathcal{E}_j (t),
\end{equation}
where $\mathcal{E}_j (t) : = \min_{1 \leq i \leq N}  \{ \exp
(-\beta(\xi_{ij}(t) -\rho) - \beta X_i^0(t-1)  )  \} $. Since
$\xi_{ij}(t)$ are Gumbel $G( \rho,\beta)$-distributed, $\exp
(-\beta(\xi_{ij}(t) -\rho)  ) $ are exponentially distributed with
parameter one. Hence, conditionally on $\F_{t-1}$,
\[
\exp \bigl(-\beta\bigl(\xi_{i j}(t) -\rho\bigr) - \beta
X_i^0(t-1) \bigr),\qquad 1 \leq i \leq N %
\]
are independent and $\exp (-\beta(\xi_{i j}(t) -\rho) - \beta
X_i^0(t-1)  )$ is distributed according to an exponential random
variable with parameter $
\exp ( \beta X_i^0(t-1)  )$. Applying the stability property of
the exponential law under independent minimum, we obtain that
conditionally on $\F_{t-1}$ each variable $\mathcal{E}_i (t)$ is
exponentially distributed with parameter one and, moreover, that the
whole vector $ \mathcal{E} (t): =  (\mathcal{E}_i (t), i \leq N
) $ is conditionally independent. Therefore, the vector $\mathcal{E}
(t)$ is independent from $\F_{t-1}$ and its coordinates $\mathcal{E}_i
(t),   1 \leq i \leq N$ are i.i.d. having an exponential law with
parameter one. Using once again the stability property of the
exponential law under independent minimum,
%
%e3.5 #&#
\begin{eqnarray}
\label{equa:defi.coef.fitness.gumbel} \eta_i (t) & :=& \P \bigl( \xi_{ij} (t+1)
+ X_i(t) > \xi_{k
j}+X_k(t), \mbox{ for
every } k \neq i | \F_t \bigr)
\nonumber
\\
& =& \P \Bigl( e^{ -\beta( \xi_{ij} (t+1) -\rho)} e^{ -
\beta X_i(t) } < \min_{k \neq i }
e^{-\beta( \xi_{k j} (t+1) -\rho)} e^{ -\beta X_k(t) } | \F_t \Bigr)
\\
& =& \exp \bigl(\beta X_i(t) \bigr) \Big/ \Biggl( \sum
_{k=1}^{N} \exp \bigl( \beta X_k(t)
\bigr) \Biggr).\nonumber
\end{eqnarray}
Then, from (\ref{equa:decomposition,gumbel,case}) we obtain that
%
%e3.6 #&#
\begin{equation}
\eta_i (t) = \mathcal{E}^{-1}_i (t) \Big/
\Biggl( \sum_{k=1}^{N}
\mathcal{E}^{-1}_k (t) \Biggr),
\end{equation}
which proves (\ref{equa:defi.coef.fitness.gumbel.final}). In
particular, the family sizes $\nu(1), \nu(2), \ldots$ have the same
distribution. If at $t=0$ the particles are distributed according to
the invariant measure the same argument holds and $\nu(t),   t \geq0$
have the same distribution.

\emph{We now prove that the $\nu(t)$'s are independent}. It suffices to
show that
%
%e3.7 #&#
\begin{equation}
\label{equa:aim:indep:prop.indep.offspring.generations.gumbel} \E \bigl[ f_1 \bigl(\nu(1) \bigr) \cdots
f_{t+1} \bigl(\nu(t+1) \bigr) \bigr] = \E \bigl[ f_1
\bigl( \nu(1) \bigr) \cdots f_{t} \bigl( \nu(t) \bigr) \bigr] \E
\bigl[ f_{t+1} \bigl(\nu(t+1) \bigr) \bigr],
\end{equation}
for all continuous bounded functions $f_1(\cdot), \ldots, f_{t}(\cdot),
f_{t+1}(\cdot)$. Let $A_{i,j;t}$ be the event
\[
A_{i,j;t} = \Bigl\{ \xi_{ji}(t+1) + X_j(t) >
\max_{ k \neq i} \bigl\{ \xi _{k i}(t+1) +
X_k(t) \bigr\} \Bigr\}
\]
that $X_i(t+1)$ descends from $X_j(t)$. Denote by $\G_t$ the $\sigma
$-algebra generated by $\F_t$ and $A_{i,j;t}$ for every $1\leq i,j\leq
N$, then $\nu(1), \ldots, \nu(t)$ are $\G_t$-measurable. We claim that
$\nu(t+1)$ is independent from $\G_t$, which proves (\ref
{equa:aim:indep:prop.indep.offspring.generations.gumbel}). Since $\nu
(t+1)$ is completely determined by $ \{\mathcal{E}_k (t+1),   1\leq
k \leq N  \} $ and $ \{\xi_{kl}(t+2),   1\leq k, l\leq N  \}
$, it is immediate that it is independent from $\F_t$. Hence, we prove
the claim once we show that $\nu(t+1)$ and $A_{i,j;t}$ are independent
for every $1\leq i,j\leq N$. Since
\[
A_{i,j;t} \in\sigma \bigl\{\F_{t}; \bigl\{
\xi_{k i}(t+1); 1 \leq k \leq N \bigr\} \bigr\} \subset
\F_{t+1}, %
\]
it suffices to show that $A_{i,j;t}$ is independent from $\sigma \{
\mathcal{E}_k (t+1),   1\leq k \leq N  \}$. It is not hard to show
that $\mathcal{E}_k (t+1)$ and $A_{i,j;t}$ are independent, whenever
$k\neq i$ and we leave the details to the reader. Let $g(\cdot)$ be a
bounded continuous function. Conditionally, on $\F_t$, $\mathcal
{E}_i(t+1)$ is the minimum of $N$ independent random variables
exponentially distributed with parameters $\exp ( \beta X_k^0(t-1)
 )$ and the set $A_{i,j;t}$ is the event that the minimum is
attained by $\exp (-\beta(\xi_{ji}(t) -\rho) - \beta X^0_j(t)
)$. Then, using standard properties of exponential distributions, we obtain
\begin{eqnarray*}
&&\E \bigl[ g \bigl(\mathcal{E}_i(t+1) \bigr) \mathbf{1}_{A_{i,j;t} }
|\F _{t} \bigr]\\
&&\quad = \P ( A_{i,j;t} |\F_{t} ) \int
_{\R_+} g(y) \cdot \frac{ \exp ( -y \sum e^{\beta X^0_k(t-1)}  )}{\sum e^{\beta
X^0_k(t-1) } } \cdot\mathrm{d}y
\\
&&\quad = \P ( A_{i,j;t} |\F_{t} ) \int_{\R_+}
\,\mathrm{d}y\, g(y) \exp( -y).
\end{eqnarray*}
We used that $X^0$ is the process shifted by $\Phi$, which satisfies
$\sum e^{\beta X^0_k(t-1)} = 1$. Then $\mathcal{E}_i(t+1)$ and
$A_{i,j;t}$ are independent, which proves the claim and, therefore, the
proposition.
\end{pf}

\begin{pf*}{Proof of Theorem~\ref{thmm.genealogy.Gumbel.case}} By
Proposition~\ref{prop.indep.offspring.generations.gumbel}, the family
sizes $\nu(t)$ are independent and identically distributed for $t \geq
1$ (and $t\geq0$ if the initial position of particles is distributed
according to the invariant measure). Furthermore, it is easy to compute
the tail distribution of $\mathcal{E}^{-1}_i (t)$
\[
\P\bigl( \mathcal{E}^{-1}_i (t) \geq x\bigr) = 1 -
e^{- x^{-1}} \sim1/x,\qquad x \to\infty,
\]
where ``$\sim$'' means that the ratio of the sides approaches to one as
$x \to\infty$, so (\ref{equa.cond.Y.alpha}) holds with $\alpha= 1$.

If $T_0< T$ and $N$ is sufficient large such that $(T-T_0)(\log N) \geq
1$, then the family sizes $\nu(t),   t \in \{\lfloor(T-T_0)(\log
N)\rfloor, \ldots, \lfloor T(\log N)\rfloor \}$ are i.i.d. It is
then possible to apply Theorem~\ref
{thmm.general.criteria.independent.case} with $\alpha=1$, which
concludes the proof.
\end{pf*}%\qed

%s4 #&#
\section{Proof of Theorem \texorpdfstring{\protect\ref
{thmm.general.criteria.independent.case}}{1.2}}\label{sec.proof.theorem}

The proof of Theorem~\ref{thmm.general.criteria.independent.case} will
be divided in two main parts. In the first one, we focus on the case
where $Y_1$ has finite second moment, which generalize $\alpha> 2$ in
(\ref{equa.cond.Y.alpha}). The proof of the first part of Theorem~\ref
{thmm.general.criteria.independent.case} is an adaptation of the proof
of part (a) of Theorem~4 in \cite{Schweinsberg2003}. In the second
part, we prove Theorem~\ref{thmm.general.criteria.independent.case} for
$\alpha\leq2$. We do so by studying the Laplace transform of $Y_i$
and its derivatives.

Before proving Theorem~\ref{thmm.general.criteria.independent.case}, we
prove a general statement about multinomial distributions. In the next
lemma, we will denote by $\nu$ a $N$-class multinomial random variable
with $N$ trials and by $\eta_i$ the probability outcomes, that are not
necessarily $N$-exchangeable.

%le4.1 #&#
\begin{lem} \label
{lem:representatation.offspring.vector.in.term.of.eta} Let $\nu= (\nu
_1, \ldots, \nu_N  )$ be a doubly stochastic multinomial random
variable with probability outcomes $\eta_1, \ldots, \eta_N$. Let also
$b_1 \geq\cdots\geq b_a \geq1 $ and $b=b_1 + \cdots+b_a$ (we also
assume that $b \leq N$). Then
%
%e4.1 #&#
\begin{equation}
\label{equa:representatation.offspring.vector.in.term.of.eta} \E \bigl[ (\nu_{1})_{b_1} \cdots(
\nu_{a})_{b_a} \bigr] = (N)_b \E \bigl[
\eta_1^{b_1} \cdots\eta_a^{b_a}
\bigr].
\end{equation}
\end{lem}
\begin{pf} To simplify the notation, we assume that $\eta_1, \ldots, \eta
_N$ are non-random. Then, $\nu$ is distributed according to a standard
multinomial distribution.
%
%e4.2 #&#
\begin{eqnarray}
\label{equa:lem:representatation.offspring.vector.in.term.of.eta}
&&\E \bigl[ (\nu_{1})_{b_1} \cdots(
\nu_{a})_{b_a} \bigr]
\nonumber
\\[-8pt]
\\[-8pt]
\nonumber
%&= \sum_{i_1+\cdots+i_a \leq N} \P\big( (\nu_{1} , \ldots\nu_{a} ) =
%(i_1,\ldots, i_a) \big) (i_{1})_{b_1} \ldots(i_{a})_{b_a} \nonumber\\
&&\quad =\mathop{ \sum_{i_j \geq b_j }}_{i_1+\cdots+i_a \leq N}
\frac{N!
\eta_1^{i_1} \cdots\eta_a^{i_a} (1-\eta_{1,\ldots,a})^{N-i_{1,\ldots
,a}}}{i_1! \cdots i_a! (N-i_{1,\ldots,a})!} \cdot\frac{i_{1}
!}{(i_1-b_1)!} \cdots\frac{i_{a} !}{(i_a-b_a)!},
\end{eqnarray}
where $i_{1,\ldots,a} := i_1+\cdots+i_a$ and $\eta_{1,\ldots,a} := \eta
_1+\cdots+\eta_a$. By a change of variables $k_j= i_j -b_j$ we rewrite
(\ref{equa:lem:representatation.offspring.vector.in.term.of.eta})
\begin{eqnarray*}
&&\sum_{k_1+\cdots+k_a \leq N-b} \frac{N!}{k_1! \cdots k_a!
(N-b-k_{1,\ldots,a})!}\cdot
\eta_1^{k_1+b_1} \cdots\eta_a^{k_a+b_a} (1-
\eta_{1,\ldots,a})^{N-b-k_{1,\ldots,a}}
\\
&&\quad = (N)_b \eta_1^{b_1} \cdots
\eta_a^{b_a} \sum\frac
{(N-b)!}{k_1! \cdots k_a! (N-b-k_{1,\ldots,a})!}\cdot
\eta_1^{k_1} \cdots\eta_a^{k_a} (1-
\eta_{1,\ldots,a})^{N-b-k_{1,\ldots,a}}
\\
&&\quad = (N)_b \eta_1^{b_1} \cdots
\eta_a^{b_a} \bigl( \eta_1 + \cdots +
\eta_a + (1-\eta_{1,\ldots,a}) \bigr)^{N-b},
\end{eqnarray*}
proving the result in the non-random case. The random case is obtained
by conditioning on $\sigma\{\eta_1, \ldots, \eta_N\}$.
\end{pf}

%s4.1 #&#
\subsection{Convergence to Kingman's coalescent \texorpdfstring{$\E[Y_1^2]<\infty$}{E[Y12]<infty}}

In \cite{Mohle2000}, M\"ohle shows that if the family sizes are not
``too large'' the processes $\Pi^{N,n}_{\lfloor t/c_N \rfloor}$ converge
to the Kingman's $n$-coalescent.

%pr4.2 #&#
\begin{prop}[(M\"ohle \cite{Mohle2000})] \label
{prop.mohle.kingman.coalescent} Suppose that
%
%e4.3 #&#
\begin{equation}
\label{equa:prop.mohle.kingman.coalescent} \lim_{N \to\infty} \frac{\E [(\nu_i)_3  ]}{N^2 c_N} = 0.
\end{equation}
Then, as $N \to\infty$, the processes $\Pi^{N,n}_{\lfloor t/c_N \rfloor
}$ converge to the Kingman's n-coalescent.
\end{prop}

We will use Proposition~\ref{prop.mohle.kingman.coalescent} to prove
Theorem~\ref{thmm.general.criteria.independent.case} in the case where
the $Y_i$'s are square integrable. We first estimate $c_N$, the
probability that two individuals have a common ancestor one generation
backward in time.

%le4.3 #&#
\begin{lem} \label{lem:cN.Y.square.integrable} Assume that the
hypotheses of Theorem~\ref{thmm.general.criteria.independent.case} hold
with $\E[Y_1^2]< \infty$ and let $c_N$ be as in (\ref{equa:defi.cN}). Then
%
%e4.4 #&#
\begin{equation}
\label{equa:lem:cN.Y.square.integrable} \lim_{N \to\infty} N c_N =
\frac{\E[Y_1 ^2 ]}{ \E[Y_1 ]^2}.
\end{equation}
\end{lem}
\begin{pf} From Lemma~\ref
{equa:representatation.offspring.vector.in.term.of.eta}, we obtain that
\[
N c_N = N^2 \E \bigl[ \eta_1
^2 \bigr]. %
\]
Let $\delta_1 > 0$, then by definition of $\eta_1$
%
%e4.5 #&#
\begin{equation}
\label{equa:lem:cN.Y.square.integrable.dominated.conv} N^2 \E \bigl[ \eta_1 ^2
\bigr] = \E \biggl[ \frac{Y_1^2}{ ( N^{-1}
\sum_{j=1}^{N} Y_j  )^2} \biggr] \geq\E \biggl[ \frac{Y_1^2}{
\delta_1 +  ( N^{-1} \sum_{j=1}^{N} Y_j  )^2}
\biggr].
\end{equation}
Since $Y_1 >0$, we use dominated convergence in (\ref
{equa:lem:cN.Y.square.integrable.dominated.conv}) to obtain that
\[
\liminf_{N \to\infty} N c_N \geq\frac{\E[Y_1^2 ]}{ \delta_1 +
 ( \E[Y_1 ]  )^2}.
\]
The inequality holds for every $\delta_1$ positive, which implies that
the above $\liminf$ is larger than $\E[Y_1 ^2 ] / \E[Y_1 ]^2$. We now
obtain an upper bound for the $\limsup$. We use the Markov inequality
to obtain that for all $c>0$
%
%e4.6 #&#
\begin{equation}
\label{equa.markov.inequality.Y.square.integrable} \lim_{x \to\infty} x^2 \P(Y_1
\geq c x) = 0.
\end{equation}
Let $S_{2,N} = \sum_{i=2}^{N} Y_i  $ and take $0<\delta_2<\E[Y_1]$
sufficiently small such that
%
%e4.7 #&#
\begin{equation}
\label{equa;choice.delta2epsilon} \frac{\E[Y_1^2]} { ( \E[Y_1] - \delta_2  )^2} \leq\frac{\E
[Y_1^2] }{ \E[Y_1]^2 } + \varepsilon/3,
\end{equation}
for a fixed $\varepsilon> 0$. Then we write
%
%e4.8 #&#
\begin{eqnarray}
\label{equa:lem:cN.Y.square.integrable.upper.bound} N^2 \E \bigl[ \eta_1 ^2
\bigr] & =& \E \biggl[ \frac{Y_1^2}{  (
N^{-1} Y_1 + N^{-1} S_{2,N}  )^2} ; S_{2,N} \geq N \bigl(
\E[Y_1] - \delta_2\bigr) \biggr]
\nonumber
\\
&&{} + \E \biggl[ \frac{Y_1^2}{  ( N^{-1} Y_1 + N^{-1}
S_{2,N}  )^2} ; S_{2,N} \leq N \bigl(
\E[Y_1] - \delta_2\bigr) \biggr]
\\
& = &\textsc{(I)} +\textsc{(II)}.\nonumber
\end{eqnarray}
Since $Y_i > 0$, we may bound $\textsc{(II)}$ in (\ref
{equa:lem:cN.Y.square.integrable.upper.bound}) as follows:
\begin{eqnarray*}
\textsc{(II)} & \leq& \E \biggl[ \frac{Y_1^2}{  ( N^{-1} Y_1
)^2} ; S_{2,N} \leq N
\bigl( \E[Y_1] - \delta_2 \bigr) \biggr]
\\
& = &N^2 \P \bigl( S_{2,N} \leq N \bigl(
\E[Y_1] - \delta_2 \bigr) \bigr).
\end{eqnarray*}
So we apply the Chernoff inequality to conclude that if $\delta_2$ is
fixed and $N$ sufficiently large, then $\textsc{(II)}$ is smaller than
$ \varepsilon/ 3 $.
\begin{eqnarray*}
\textsc{(I)} & \leq&\E \biggl[ \frac{Y_1^2}{  ( \E[Y_1] - \delta_2
 )^2} ; Y_1 \leq N
\bigl(\E[Y_1] - \delta_2 \bigr) \biggr]\\
&&{} +
N^2 \P \bigl( Y_1 \geq N \bigl(\E[Y_1] -
\delta_2 \bigr) \bigr)
\\
& \leq&\E \biggl[ \frac{Y_1^2}{  ( \E[Y_1] - \delta_2  )^2 } \biggr] + N^2 \P \bigl(
Y_1 \geq N \bigl( \E[Y_1] - \delta_2
\bigr) \bigr).
\end{eqnarray*}
From (\ref{equa.markov.inequality.Y.square.integrable}) with $c= \E
[Y_1]- \delta_2$, the second term in the right-hand side converges to
zero as $N\to\infty$, and we may choose $N$ conveniently such that it
is smaller than $\varepsilon/3$. It is implied that $N$ is taken such
that $\textsc{(II)}$ is also smaller than $\varepsilon/ 3$. Then,
applying the upper bounds in (\ref
{equa:lem:cN.Y.square.integrable.upper.bound}) we obtain
\[
N^2 \E \bigl[ \eta_1 ^2 \bigr] \leq
\frac{ \E [Y_1^2 ]}{
 ( \E[Y_1] - \delta_2  )^2} + \frac{2}{3} \cdot\varepsilon< \frac{ \E [Y_1^2 ]}{ \E[Y_1]^2} +
\varepsilon. %
\]
Since the inequality holds for every $\varepsilon>0$ and $N$ large
enough, we conclude that $\limsup N c_N \leq\E[Y_1^2]/\E[Y_1]^2$
proving the lemma.
\end{pf}

\begin{pf*}{Proof of Theorem~\ref{thmm.general.criteria.independent.case} in
the case $\E[Y_1^2]< \infty$} In order to prove Theorem~\ref
{thmm.general.criteria.independent.case}, it suffices to show that (\ref
{equa:prop.mohle.kingman.coalescent}) holds and apply Proposition~\ref
{prop.mohle.kingman.coalescent}. From Lemma~\ref
{lem:cN.Y.square.integrable}, there exists a constant $c < 1$ such that
for $N$ sufficiently large $ N c_N > c  \E[Y_1 ^2 ] /\E[Y_1]^2 $, hence
\[
0 \leq\frac{\E [ (\nu_1)_3  ]}{N^2 c_N} \leq\frac{\E [ (\nu
_1)_3  ]}{N} \cdot \frac{ \E[Y_1 ]^2}{ c \E[Y_1 ^2 ]}.
\]
Then, to prove the convergence in (\ref
{equa:prop.mohle.kingman.coalescent}), it suffices to show that $N^{-1}
\E [ (\nu_1)_3  ] \to0 $. From (\ref
{equa:representatation.offspring.vector.in.term.of.eta}), it is
equivalent to $N^2 \E [ \eta_1^3  ] \to0$ as $N \to\infty$. We
proceed as in (\ref{equa:lem:cN.Y.square.integrable.upper.bound}) and obtain
%
%e4.9 #&#
\begin{eqnarray}
N^2 \E \bigl[ \eta_1 ^3 \bigr] & =&
N^2 \E \biggl[ \frac{Y_1^3}{  (
Y_1 + S_{2,N}  )^3} ; S_{2,N} \geq N \bigl(
\E[Y_1] - \delta_2\bigr) \biggr]
\nonumber
\\
&&{} + N^2 \E \biggl[ \frac{Y_1^3}{  ( Y_1 + S_{2,N}
)^3} ; S_{2,N} \leq N
\bigl(\E[Y_1] - \delta_2\bigr) \biggr]
\\
& = &\textsc{(I)} +\textsc{(II)}.\nonumber
\end{eqnarray}
Applying the same argument of Lemma~\ref{lem:cN.Y.square.integrable},
we conclude that $\textsc{(II)}$ converges to zero as $N$ diverges and
we also obtain the following upper bound to $\textsc{(I)}$
%
%e4.10 #&#
\begin{equation}
\label{equa:bound.I.square.integrable} \textsc{(I)} \leq N^2 \E \biggl[ \frac{Y_1^3}{ ( N (\E[Y_1]-\delta_2)
 )^3} ;
Y_1 \leq N \bigl(\E[Y_1]-\delta_2 \bigr)
\biggr] + N^2 \P \bigl( Y_1 \geq N \bigl(
\E[Y_1]-\delta_2 \bigr) \bigr).
\end{equation}
We use the Markov inequality to show that the second term in the
right-hand side of (\ref{equa:bound.I.square.integrable}) converges to
zero as $N \to\infty$. As a consequence, to finish the proof it
suffices to show that the first term in the right-hand side of (\ref
{equa:bound.I.square.integrable}) converges to zero as $N \to\infty$.
For $\varepsilon> 0$ let $L \in\R_+$ be such that
\[
\E\bigl[Y_1^2 ; Y_1 \geq L\bigr] /
\bigl( \E[Y_1]-\delta_2 \bigr)^2 <
\varepsilon/ 2. %
\]
Since $L$, $\delta_2$ and $\varepsilon$ are fixed we may choose $N$
sufficiently large such that
\[
\frac{ L \E[Y_1^2] }{ N  ( \E[Y_1]-\delta_2  ) ^3 } < \varepsilon / 2, %
\]
and we bound the first term in the right-hand side of (\ref
{equa:bound.I.square.integrable})
%
%e4.11 #&#
\begin{eqnarray}
&&N^{2}  \E \biggl[ \frac{Y_1^3}{ ( N \E[Y_1]-\delta_2  )^3} ; Y_1 \leq N
\bigl(\E[Y_1]-\delta_2 \bigr) \biggr]
\nonumber
\\
&&\quad\leq\frac{L}{N  (\E[Y_1]-\delta_2  ) ^3} \cdot\E \bigl[ Y_1^2 ;
Y_1 \leq L \bigr] + \frac{\E [ Y_1^2 ;   L \leq Y_1 \leq
N (\E[Y_1]-\delta_2 )  ]}{ ( \E[Y_1]-\delta_2  ) ^2}
\\
&&\quad\leq\frac{L}{N  (\E[Y_1]- \delta_2  ) ^3} \cdot\E \bigl[ Y_1^2 \bigr] +
\frac{\E [ Y_1^2 \mathbf{1}_{ \{ Y_1 \geq L \}}
]}{ (\E[Y_1]- \delta_2  ) ^2} < \varepsilon,\nonumber
\end{eqnarray}
that finishes the proof.
\end{pf*} %\qed

%s4.2 #&#
\subsection{Proof of Theorem \texorpdfstring{\protect\ref
{thmm.general.criteria.independent.case}}{1.2} when \texorpdfstring{$\alpha\leq2$}{alpha<=2}}

The strategy to prove Theorem~\ref
{thmm.general.criteria.independent.case} in the case $\alpha\leq2$ is
to compute the limits (\ref{equa.thmm:mohle.sagitov}) and apply Theorem~\ref{thmm:mohle.sagitov}. In the next proposition, we show how the
moments of $\eta_i$'s are related to the Laplace transform of $Y_i$.

%pr4.4 #&#
\begin{prop}\label{prop:integral.representation.b's.moment} Let $b_1
\geq b_2 \geq\cdots\geq b_a \geq2 $ be positive integers, $b = b_1+
\cdots+ b_a$ and for $1 \leq i \leq N$
\[
\eta_i : = \frac{Y_i}{\sum_{i=1}^{N} Y_j}, %
\]
where $Y_1, \ldots, Y_N$ are i.i.d. random variables. Then
%
%e4.12 #&#
\begin{equation}
\label{equa.lem:integral.representation.b's.moment} \E \bigl[ \eta_1^{b_1} \cdots
\eta_a^{b_a} \bigr] = \frac{1}{\Gamma(b)} \int
_{0}^{\infty} u^{b-1} I_0(u)^{N-a}
I_{b_1}(u) \cdots I_{b_a}(u) \,\mathrm{d} u ,
\end{equation}
where $\Gamma(\cdot)$ is the gamma function and
%
%e4.13 #&#
\begin{equation}
\label{equa:definition.I_p.general} I_{p}(u) = \E \bigl[ Y_1^{p}
e^{-u Y_1 } \bigr], \qquad p \in\N.
\end{equation}
\end{prop}

\begin{pf} For every $z\in\R_+^*$ we have the following integral representation
%
%e4.14 #&#
\begin{equation}
\label{equa:moment.representation.gamma} z^{-b}= \frac{1}{\Gamma(b)}\int_{0}^{\infty}
u^{b-1}e^{-u z} \,\mathrm{d} u ,
\end{equation}
then applying (\ref{equa:moment.representation.gamma}) with $z=\sum_{i=1}^{N} Y_i $ we obtain
%
%e4.15 #&#
\begin{eqnarray}
\label{equa:integral.representation.p.moment} \E \bigl[ \eta_1^{b_1} \cdots
\eta_a^{ b_a} \bigr] & =& \E \biggl[ Y_1^{b_1}
\cdots Y_a^{b_a} \frac{1}{\Gamma(b)} \int
_{0}^{\infty} u^{b-1} e^{-u \sum_{i=1}^{N} Y_i}
\,\mathrm{d} u \biggr]
\nonumber
\\
& = &\int_{0}^{\infty} \frac{u^{b-1}}{\Gamma(b)} \E \bigl[
Y_1^{b_1} \cdots Y_a^{b_a}
e^{-u \sum_{i=1}^{N} Y_i} \,\mathrm{d} u \bigr] \qquad\mbox{(Fubini)}
\\
& =& \int_{0}^{\infty} \frac{u^{b-1}}{\Gamma(b)} \E \bigl[
\exp ( - u Y_1 ) \bigr]^{N-a} \prod
_{i=1}^{a}\E \bigl[ Y_1^{b_i}
\exp (-u Y_1 ) \bigr] \,\mathrm{d} u.\nonumber
\end{eqnarray}
In the last equality, we used the fact that $Y_i$ are i.i.d. Hence,
from the definition of $I_{b_i}$ we obtain that (\ref
{equa:integral.representation.p.moment}) and (\ref
{equa.lem:integral.representation.b's.moment}) are equal, proving the result.
\end{pf}

It is clear that the functions $I_{p}(u)$ are decreasing and attain
their maximum at zero. Moreover, the following relation can be easily deduced
\[
\frac{\mathrm{d}^{p}}{\mathrm{d}u^{p}} I_0(u) = (-1)^{p} I_{p}
(u). %
\]

We now outline the strategy to prove Theorem~\ref
{thmm.general.criteria.independent.case}.
\begin{longlist}[1.]
\item[1.] We first obtain a precise asymptotic of $I_{p}(u)$ in the
neighborhood of zero, where $I_{p}(u)$ attains its maximum. As the
reader will see, the behavior of $I_{p}(u)$ depends on $\alpha$ and
each case will be studied separately.\vadjust{\goodbreak}
\item[2.] We show that the integral in the right-hand side of (\ref
{equa.lem:integral.representation.b's.moment}) is essentially
determined by the immediate neighborhood of zero.
\item[3.] We estimate $\E [ \eta_1^{b_1} \cdots\eta_a^{b_a}  ]$.
\item[4.] We prove Theorem~\ref{thmm.general.criteria.independent.case}
using Lemma~\ref{lem:representatation.offspring.vector.in.term.of.eta}
that relates (\ref{equa.thmm:mohle.sagitov}) with $\E [ \eta_1^{b_1}
\cdots\eta_a^{b_a}  ]$.
\end{longlist}

%We follow the strategy presented above and we start by studying the
%behavior of $I_{b_i}(u)$ in the neighborhood of zero.

%le4.5 #&#
\begin{lem}\label{lem:asymptotics.I.p.neighborhood.of.zero.alpha} Let
$I_\cdot(u)$ be given by (\ref{equa:definition.I_p.general}).
\begin{longlist}[(a)]
\item[(a)] If $Y_i$ satisfies (\ref{equa.cond.Y.alpha}) with $\alpha= 2$
and $C=1$. Then
\begin{eqnarray*}
\label{equa:asymptotics.I.p=0.alpha=2} %
I_0 (u) &=& 1 -u \E
[Y_1 ] + o(u) \qquad \mbox{when } u \to 0^+;
\\
I_{2} (u)& =& (-2 \log u ) + o \bigl(\log\bigl(u^{-1}\bigr)
\bigr) \qquad \mbox{when } u \to0^+;
\\
I_{p} (u)& =& u^{2-p} \bigl( 2 \Gamma(p -2) \bigr) + o
\bigl(u^{2-p}\bigr)\qquad \mbox{when } p \geq3 \mbox{ and } u \to0^+.
\end{eqnarray*}

\item[(b)] When $Y_i$ satisfies (\ref{equa.cond.Y.alpha}) with $1<\alpha< 2$
and $C=1$. Then
\begin{eqnarray*}
\label{equa:asymptotics.I.p=0.1<alpha<2} %
I_0 (u)& =& 1 -u \E
[Y_1 ] + o(u)\qquad \mbox{when } u \to 0^+;
\\
I_{p} (u) &= & u^{\alpha- p} \bigl( \alpha \Gamma(p -\alpha) \bigr) +
o\bigl(u^{\alpha- p}\bigr) \qquad \mbox{when } p \geq2 \mbox{ and } u \to0^+.
\end{eqnarray*}

\item[(c)] If (\ref{equa.cond.Y.alpha}) holds with $\alpha= 1$ and $C=1$. Then
\begin{eqnarray*}
\label{equa:asymptotics.I.p=0.alpha=1} %
I_0 (u) &= &1 + ( u \log
u ) + o(u \log u),\qquad  \mbox{when } u \to0^+;
\\
I_{p} (u)& = &u^{ 1 - p} \Gamma(p - 1) + o\bigl(u^{ 1 - p}
\bigr), \qquad \mbox {when } p \geq2 \mbox{ and } u \to0^+.
\end{eqnarray*}

\item[(d)] Assume that $Y_i$ satisfies (\ref{equa.cond.Y.alpha}) with
$0<\alpha<1$ and $C=1$. Then
\begin{eqnarray*}
\label{equa:asymptotics.I.p=0.alpha<1} %
I_0 (u) &=& 1 -
u^{\alpha} \Gamma(1-\alpha) + o\bigl(u^{\alpha}\bigr)\qquad\mbox {when
} u \to0^+;
\\
I_{p} (u)&=& u^{ \alpha- p} \bigl( \alpha \Gamma(p - \alpha) \bigr)
+ o\bigl(u^{ \alpha- p}\bigr)\qquad \mbox{when } p \geq2 \mbox{ and } u \to0^+.
\end{eqnarray*}
\end{longlist}
\end{lem}

\begin{pf} See Appendix~\ref{appendix.lem:asymptotics.I.p.neighborhood.of.zero.alpha}.
\end{pf}

In the next lemma, we show that only the immediate neighborhood of zero
contributes to the integral in (\ref
{equa.lem:integral.representation.b's.moment}) of Proposition~\ref
{prop:integral.representation.b's.moment}.

%le4.6 #&#
\begin{lem} \label{lem:integral.upper.tail.alpha.leq.2} Let $I_\cdot
(u)$ be given by (\ref{equa:definition.I_p.general}) and $ \kappa_N :=
(\log N)^2/ N $, assume also that $Y_i$ satisfies (\ref
{equa.cond.Y.alpha}) with $\alpha\leq2$ and $C=1$. Then, for every $K
\in\N$
%
%e4.16 #&#
\begin{equation}
\label{equa.lem:integral.upper.tail.alpha=2} \lim_{N \to\infty} N^K \int
_{ \kappa_N}^{\infty} u^{b-1} I_0(u)^{N-a}
I_{b_1}(u) \cdots I_{b_a}(u) \,\mathrm{d} u =0,
\end{equation}
where $b_1 \geq\cdots\geq b_a$ are fixed integers and $b= b_1 +
\cdots+ b_a$. Hence, the integral in (\ref
{equa.lem:integral.upper.tail.alpha=2}) decreases faster than any
polynomial in $N$.
\end{lem}

\begin{pf} Since $I_0$ is a decreasing function
%
%e4.17 #&#
\begin{eqnarray}
\label{equa:bound.cov.moments.away.from.zero.alpha=2} &&\int_{ \kappa_N}^{\infty}  u^{b-1}
I_0(u)^{N-a} I_{b_1}(u) \cdots
I_{b_a}(u) \,\mathrm{d} u
\nonumber
\\
&&\quad\leq I_0 ( \kappa_N)^{N-a} \int
_{ \kappa_N}^{\infty} u^{b-1} I_{b_1}(u)
\cdots I_{b_a}(u) \,\mathrm{d} u
\nonumber
\\[-8pt]
\\[-8pt]
\nonumber
&&\quad \leq I_0 ( \kappa_N)^{N-a} \int
_{0}^{\infty} u^{b-1} \E
\bigl[Y_1^{b_1}e^{-u Y_1} \bigr] \cdots\E \bigl[
Y_a^{b_a}e^{-u Y_a} \bigr] \,\mathrm{d} u
\\
&&\quad = I_0 ( \kappa_N )^{N-a} \Gamma(b) \E \biggl[
\frac{Y_1^{b_1} \cdots
Y_a^{b_a}}{ ( \sum_{i=1}^{a} Y_i  )^b} \biggr].\nonumber
\end{eqnarray}
In the last equality, we proceed as in Proposition~\ref
{prop:integral.representation.b's.moment} and use the integral
representation (\ref{equa:moment.representation.gamma}) with $z = \sum_{i=1}^{a} Y_i$. The expected value in the right-hand side of (\ref
{equa:bound.cov.moments.away.from.zero.alpha=2}) is bounded from above
by one. Applying Lemma~\ref
{lem:asymptotics.I.p.neighborhood.of.zero.alpha} with $u = \kappa_N \to
0^+$ as $N \to\infty$
\begin{eqnarray*} %I_0 ( \kappa_N )^{N-a} = \exp- \E[Y_i](\log N)^2 + o (\log^2 N), &
% \mbox{if } \alpha=2; \\
I_0 (
\kappa_N )^{N-a} &=& \exp \bigl\{ - \E[Y_i](\log
N)^2 + o \bigl(\log^2 N\bigr) \bigr\}\qquad \mbox{if } 1<
\alpha\leq2;
\\
I_0 ( \kappa_N )^{N-a}& =& \exp \bigl\{ - (\log
N)^3 +(\log N)^2 (\log2 \log N ) + o \bigl(
\log^3 N\bigr) \bigr\}\qquad \mbox{if } \alpha=1;
\\
I_0 ( \kappa_N )^{N-a}& =& \exp \bigl\{ -
\Gamma(1-\alpha) N^{1-\alpha} (\log N)^{2\alpha} + o
\bigl(N^{1-\alpha} (\log N)^{2\alpha}\bigr) \bigr\}\qquad \mbox{if } 0<
\alpha<1;
\end{eqnarray*}
that decreases faster than any polynomial in $N$.
\end{pf}

The $\kappa_N$ in Lemma~\ref{lem:integral.upper.tail.alpha.leq.2} is
not optimal. The reason we have chosen such $\kappa_N$ will be clear in
the proof of Proposition~\ref
{prop:integral.representation.b's.moment.assymptotic.alpha} below,
where we estimate $ \E [ \eta_1^{b_1} \cdots\eta_a^{b_a}  ]$.

%pr4.7 #&#
\begin{prop} \label
{prop:integral.representation.b's.moment.assymptotic.alpha} Let $b_1
\geq b_2 \geq\cdots\geq b_a \geq2 $ be positive integers, $b = b_1+
\cdots+ b_a$, and $\eta_i$ be as in Proposition~\ref
{prop:integral.representation.b's.moment}.
\begin{longlist}[(a)]
\item[(a)] Suppose $Y_i$ satisfies (\ref{equa.cond.Y.alpha}) with $\alpha=2$
and $C=1$. Let $g := \max\{i;   b_i \geq3 \}$, we adopt the
convention that $\max\{\varnothing\} = 0$. Then
%
%e4.18 #&#
\begin{eqnarray}
\label
{prop.lem:integral.representation.b's.moment.assymptotic.alpha=2} &&\lim_{N \to\infty} \E \bigl[ \eta_1^{b_1}
\cdots\eta_a^{b_a} \bigr] \cdot\frac{ N^{2 a} }{ (\log N)^{a-g}}
\nonumber
\\[-8pt]
\\[-8pt]
\nonumber
&&\quad= \Gamma(2
a)\cdot\frac{ 2^a
  \prod_{i=1}^g \Gamma(b_i-2)}{ \Gamma(b) \E[Y_1]^{2 a} }.
\end{eqnarray}

\item[(b)] If (\ref{equa.cond.Y.alpha}) holds with $1<\alpha<2$ and $C=1$. Then
%
%e4.19 #&#
\begin{eqnarray}
\label
{prop.lem:integral.representation.b's.moment.assymptotic.1<alpha<2} &&\lim_{N \to\infty} \E \bigl[ \eta_1^{b_1}
\cdots\eta_a^{b_a} \bigr] N^{a \alpha}
\nonumber
\\[-8pt]
\\[-8pt]
\nonumber
&&\quad = \Gamma(a
\alpha) \cdot\frac{ \prod_{i=1}^a \alpha
\Gamma(b_i-\alpha) }{\Gamma(b) \E[Y_1]^{a \alpha}}.
\end{eqnarray}

\item[(c)] If we assume that $Y_i$ satisfies (\ref{equa.cond.Y.alpha}) with
$\alpha=1$ and $C=1$. Then
%
%e4.20 #&#
\begin{equation}
\label
{prop.lem:integral.representation.b's.moment.assymptotic.alpha=1} \lim_{N \to\infty} \E \bigl[ \eta_1^{b_1}
\cdots\eta_a^{b_a} \bigr] (N \log N )^{a} =
\Gamma(a) \cdot\frac{ \prod_{i=1}^a \Gamma(b_i-1)
}{\Gamma(b)}.
\end{equation}

\item[(d)] If (\ref{equa.cond.Y.alpha}) holds with $0 <\alpha<1$ and $C=1$. Then
%
%e4.21 #&#
\begin{equation}
\label
{prop.lem:integral.representation.b's.moment.assymptotic.alpha<1} \lim_{N \to\infty} \E \bigl[ \eta_1^{b_1}
\cdots\eta_a^{b_a} \bigr] N^{a} = \Gamma(a) \cdot
\frac{ \alpha^{a-1} \prod_{i=1}^a \Gamma
(b_i-\alpha) }{\Gamma(1-\alpha)^a \Gamma(b)}.\vspace*{1pt}
\end{equation}
\end{longlist}
\end{prop}
\begin{pf} See Appendix~\ref{appendix.prop:integral.representation.b's.moment.assymptotic.alpha}.
\end{pf}

We now compute $c_N$ the probability that two individuals randomly
chosen have the same ancestor.

%co4.8 #&#
\begin{cor}\label{cor.assymptotic.c.N.alpha} Assume that the hypotheses
of Theorem~\ref{thmm.general.criteria.independent.case} hold and let
$c_N$ be as in (\ref{equa:defi.cN}). Assume also that the $Y_i$'s
satisfy (\ref{equa.cond.Y.alpha}) with $\alpha\leq2$ and $C=1$. Then
%
%e4.22 #&#
\begin{eqnarray}
\lim_{N\to\infty} \frac{N   c_N}{\log N}& =&
\frac{ 2 }{
\E[Y_1]^2}\qquad \mbox{if } \alpha=2;\nonumber
\\[2pt]
\lim_{N\to\infty} \frac{c_N}{N^{1-\alpha}}& =&\frac{
\alpha\Gamma(\alpha) \Gamma(2- \alpha)}{ \E[Y_1]^\alpha} \qquad \mbox
{if } 1<\alpha<2;
\\[2pt]
\lim_{N\to\infty} (\log N)c_N &=& 1 \qquad \mbox{if }
\alpha=1. \nonumber\vspace*{1pt}
\end{eqnarray}
Finally, if $Y_i$ satisfies (\ref{equa.cond.Y.alpha}) with $0< \alpha<
1$ and $C=1$, then
%
%e4.23 #&#
\begin{equation}
\label{cor.assymptotic.c.N.alpha<1} \lim_{N \to\infty} c_N =
\frac{\Gamma(2-\alpha)}{\Gamma(1-\alpha)}.\vspace*{1pt}
\end{equation}
\end{cor}
\begin{pf} It is a direct application of Lemma~\ref
{lem:representatation.offspring.vector.in.term.of.eta} and Proposition~\ref{prop:integral.representation.b's.moment.assymptotic.alpha}.
\end{pf}

\begin{pf*}{Proof of Theorem~\ref{thmm.general.criteria.independent.case} in
the cases $\alpha\leq2$} We analyze each case separately and compute
the limits
\[
\lim_{N \to\infty} \frac{\E [(\nu_1)_{b_1} \cdots(\nu_a)_{b_a}
 ]}{N^{b-a}c_N}.\vspace*{1pt} %
\]
If $\P(Y_i \geq x) \sim x^{-2}$ as $x\to\infty$, denote by $g = \max\{
i;   b_i \geq3 \}$ (as in Proposition~\ref
{prop:integral.representation.b's.moment.assymptotic.alpha}). Then, as
$N\to\infty$
\begin{eqnarray*}
&&\frac{\E [(\nu_1)_{b_1} \cdots(\nu_a)_{b_a}  ]}{N^{b-a}c_N}
\\
&&\quad = \frac{(N)_b}{N^{b-a}c_N} \cdot\E \bigl[\eta^{b_1} \cdots\eta
^{b_a} \bigr] \qquad (\mbox{Lemma~\ref
{lem:representatation.offspring.vector.in.term.of.eta}})
\\
&&\quad \sim N^a \frac{N}{ \log N} \cdot\frac{\E[Y_1]^2 }{ 2 } \cdot\E \bigl[
\eta^{b_1} \cdots\eta^{b_a} \bigr]\qquad
(\mbox{Corollary~\ref{cor.assymptotic.c.N.alpha}})\vadjust{\goodbreak}
\\
&&\quad \sim\frac{ N^{a+1}}{\log N} \cdot\frac{\E[Y_1]^2 }{2} \cdot \frac{(\log N)^{a-g}}{N^{2 a}}\cdot
\Gamma(2 a)\cdot\frac{ 2^a
\prod_{i=1}^g \Gamma(b_i-2)}{ \Gamma(b) \E[Y_1]^{2 a} }\qquad   (\mbox {Proposition~\ref
{prop:integral.representation.b's.moment.assymptotic.alpha}})
\\
& &\qquad= \frac{(\log N)^{a-g -1}}{N^{a-1}} \cdot\Gamma(2 a)\cdot\frac
{ 2^{a-1}  \prod_{i=1}^g \Gamma(b_i-2)}{ \Gamma(b) \E[Y_1]^{2 (a-1)}
},
\end{eqnarray*}
which converges to zero whenever $a \geq2$. If $a=1=g$, which implies
$b_a=b \geq3$, then
\[
\frac{\E [(\nu_1)_{b_1}  ]}{N^{b-1}c_N} \sim\frac{1}{\log N} \cdot\frac{\Gamma(b-2)}{\Gamma(b)\E[Y_1]} \to0 \qquad\mbox{as
} N\to \infty. %
\]
On the other hand, if $a=1$ and $g=0$, that is, $b=2$, then
\[
\lim_{N \to\infty} \frac{\E [(\nu_1)_{2}  ]}{N^{2-1}c_N} = 1. %
\]
Hence, in the scaling limit we may only observe collisions of two
distinct blocks that do not occur simultaneously, \textit{that is,}
Kingman's coalescent.

\emph{In the case $1<\alpha<2$ we proceed as above obtaining}
%
%\begin{eqnarray*}
%\frac{\E\left[(\nu_1)_{b_1} \cdots(\nu_a)_{b_a} \right]}{N^{b-a}c_N} &
%= \frac{(N)_b}{N^{b-a}c_N} \cdot\E\left[\eta^{b_1} \cdots\eta^{b_a}
%\right] &  (\mbox{Lemma~\ref{lem:representatation.offspring.vector.in.term.of.eta}}) \\
%& \sim N^a \cdot N^{\alpha-1} \cdot\frac{ \E[Y_1]^\alpha}{ \alpha
%\Gamma(\alpha) \Gamma(2- \alpha) } \cdot\E[\eta^{b_1} \cdots
%\eta^{b_a}] &  (\mbox{Corollary~\ref{cor.assymptotic.c.N.1<alpha<2}}) \\
%& \sim N^{a+\alpha-1} \cdot\frac{ \E[Y_1]^\alpha}{ \alpha\Gamma(
%\alpha) \Gamma(2- \alpha) } \cdot\frac{1}{N^{\alpha a}}\cdot\Gamma(
%\alpha a)\cdot\frac{  \prod_{i=1}^a \alpha\Gamma(b_i-\alpha)}{
%\Gamma(b) \E[Y_1]^{\alpha a} } &  (\mbox{Proposition~\ref{prop:integral.representation.b's.moment.assymptotic.1<alpha<2}}) \
%\
%&= \frac{1}{N^{(a-1)(\alpha-1)}} \frac{ \E[Y_1]^\alpha}{ \alpha
%\Gamma(\alpha) \Gamma(2- \alpha) } \cdot\Gamma(\alpha a)\cdot\frac{
% \prod_{i=1}^a \alpha\Gamma(b_i-\alpha)}{ \Gamma(b) \E[Y_1]^{\alpha
%a} }. &
%\end{eqnarray*}
\[
\frac{\E [(\nu_1)_{b_1} \cdots(\nu_a)_{b_a}  ]}{N^{b-a}c_N} \sim\frac{ \Gamma(\alpha a)}{N^{(a-1)(\alpha-1)}} \cdot\frac{ \E
[Y_1]^\alpha}{ \alpha\Gamma(\alpha) \Gamma(2- \alpha) } \cdot
\frac{
 \prod_{i=1}^a \alpha\Gamma(b_i-\alpha)}{ \Gamma(b) \E[Y_1]^{\alpha
a} } \qquad \mbox{as } N\to\infty,
\]
that converges to zero whenever $a \geq2$. If $a=1$ and \textit{a
fortiori} $b_a=b $
\begin{eqnarray*}
\lim_{N \to\infty} \frac{\E [(\nu_1)_{b}  ]}{N^{b-1}c_N} &=& \frac{\Gamma(b-\alpha)}{\Gamma(b) \Gamma(2-\alpha)}
\\
& =& \frac{(b-1-\alpha)\cdots(2-\alpha)}{(b-1)!}
\\
& =& \frac{B(b-\alpha,\alpha)}{B(2-\alpha, \alpha)} = \lambda_{b;b} ,
\end{eqnarray*}
where $B(c,d)= \Gamma(c) \Gamma(d)/ \Gamma(c+d)$, as defined in Theorem~\ref{thmm.general.criteria.independent.case}. Hence, using the
recursive formula (\ref{equa.recursion.lambda}) for $\lambda_{b;k}$
\begin{eqnarray*}
\lambda_{b;b-1;1} & =& \lambda_{b-1,b-1} -\lambda_{b,b}
\\
& = &\frac{\Gamma(b-1-\alpha)}{\Gamma(b-1) \Gamma(2-\alpha)} - \frac
{\Gamma(b-\alpha)}{\Gamma(b) \Gamma(2-\alpha)}
\\
& =&\frac{\alpha}{b-1} \cdot\frac{\Gamma(b-1-\alpha)}{\Gamma(b-1)
\Gamma(2-\alpha)}
\\
& = &\frac{B(b-1-\alpha,1+ \alpha)}{B(2-\alpha, \alpha)} = \lambda_{b;b-1}.
\end{eqnarray*}
We may proceed by recurrence and conclude the convergence to the
Beta-coalescent.

\emph{In the case $\alpha=1$, we have that}
%
%\begin{eqnarray*}
%\frac{\E\left[(\nu_1)_{b_1} \cdots(\nu_a)_{b_a} \right]}{N^{b-a}c_N} &
%= \frac{(N)_b}{N^{b-a}c_N} \cdot\E\left[\eta^{b_1} \cdots\eta^{b_a}
%\right] &  (\mbox{Lemma~\ref{lem:representatation.offspring.vector.in.term.of.eta}}) \\
%& \sim N^a \cdot\log N \cdot\E[\eta^{b_1} \cdots\eta^{b_a}] &
%(\mbox{Corollary~\ref{cor.assymptotic.c.N.alpha=1}}) \\
%& \sim N^{a} \cdot\log N \cdot\frac{ 1 }{ N^a \log^a N } \cdot
%\Gamma( a)\cdot\frac{ \prod_{i=1}^a \Gamma(b_i-1)}{ \Gamma(b) } &
% (\mbox{Proposition~\ref{prop:integral.representation.b's.moment.assymptotic.alpha=1}}) \\
%&= \frac{ 1 }{\log^{a-1} N } \cdot\Gamma( a)\cdot\frac{
%\prod_{i=1}^a \Gamma(b_i-1)}{ \Gamma(b) }. &
%\end{eqnarray*}
\[
\frac{\E [(\nu_1)_{b_1} \cdots(\nu_a)_{b_a}  ]}{N^{b-a}c_N} \sim\frac{ \Gamma( a) }{(\log N )^{a-1}} \cdot \frac{ \prod_{i=1}^a
\Gamma(b_i-1)}{ \Gamma(b) }\qquad \mbox{as } N
\to\infty,
\]
that converges to zero whenever $a \geq2$, implying that we do not
observe simultaneous collisions in the time scale. If $a=1$ and \textit{a
fortiori} $b_a=b $
\[
\lim_{N \to\infty} \frac{\E [(\nu_1)_{b}  ]}{N^{b-1}c_N} = \frac{\Gamma(b-1)}{\Gamma(b)}
\\
 = \frac{1}{b-1} = \int_{[0,1]} x^{b-2}
\,\mathrm{d} x.
\]
Hence, using the recursive formula (\ref{equa.recursion.lambda}) for
$\lambda_{b;k}$, we can conclude the convergence in distribution to the
Bolthausen--Sznitman coalescent.

\emph{When $\alpha<1$, by Corollary~\textup{\ref{cor.assymptotic.c.N.alpha}}
$\lim c_N >0$.} Then, as $ N\to\infty$
%$$
%\lim_{N \to\infty} \frac{\E\left[(\nu_1)_{b_1} \cdots(\nu_a)_{b_a}
%\right]}{N^{b-a}} = p_{b;b_1,\cdots,b_a}.
%$$
%Using the results of the present Section, we obtain as $ N\to\infty$
%e4.24 #&#
\begin{eqnarray}
\label{equa.alpha<1.Schweinsberg} \frac{\E [(\nu_1)_{b_1} \cdots(\nu_a)_{b_a}  ]}{N^{b-a}} & = &\frac{(N)_b}{N^{b-a}} \cdot\E \bigl[
\eta^{b_1} \cdots\eta^{b_a} \bigr] \qquad(\mbox{Lemma~\ref
{lem:representatation.offspring.vector.in.term.of.eta}})
\nonumber
\\
& \sim&\Gamma(a) \cdot\frac{ \alpha^{a-1} \prod_{i=1}^a \Gamma
(b_i-\alpha) }{\Gamma(1-\alpha)^a \Gamma(b)} \qquad (\mbox{Proposition~\ref{prop:integral.representation.b's.moment.assymptotic.alpha}})
\nonumber
\\[-8pt]
\\[-8pt]
\nonumber
& =& \frac{\alpha^{a-1} (a-1)!}{(b-1)!} \cdot\prod\frac{\Gamma
(b_i-\alpha)}{\Gamma(1-\alpha)}
\\
& =& \frac{\alpha^{a-1} (a-1)!}{(b-1)!} \cdot\prod[1-\alpha]_{b_i-1;
1}, \nonumber
\end{eqnarray}
where $[x]_{m,y} :=x (x+y)\cdots(x+(m-1)y)$. We finish the proof by
observing that the limit in (\ref{equa.alpha<1.Schweinsberg}) is
exactly the same limit that Schweinsberg obtains when studying
coalescent processes that govern the genealogical trees of
supercritical Galton--Watson processes with selection; see Section~4 of
\cite{Schweinsberg2003}. %\qed
\end{pf*}

\begin{appendix}\label{app}

%\renewcommand*{\thesection}{\Alph{section}}
%\renewcommand*{\theequation}{\Alph{section}.\arabic{equation}}

%s5 #&#
\section{Proof of Lemma \texorpdfstring{\protect\ref
{lem:asymptotics.I.p.neighborhood.of.zero.alpha}}{4.5}}\label
{appendix.lem:asymptotics.I.p.neighborhood.of.zero.alpha}

In this appendix, we present the proof of Lemma~\ref
{lem:asymptotics.I.p.neighborhood.of.zero.alpha}. We first prove the
expansion of $I_0(u)$ and then of $I_{p}(u)$ for $p \geq2$. The idea
of the proof is more or less the same for every $0<\alpha\leq2$, but
some technical adaptations are required in specific cases.

The Laplace transform $I_0$ of $Y_i$ is differentiable, when $1<\alpha
\leq2$ and $I'_0(0) = \E[Y_i]$, then in this case, the expansion of
$I_0(u)$ is obtained by a simple Taylor development at zero. For
$\alpha\leq1$, the Laplace transform of $Y_1$ is no longer
differentiable at zero. On the other hand, we have that
%
%e5.1 #&#
\begin{eqnarray}
\label{equa.I.0.important.part.alpha.leq.1} \E \bigl[ e^{-u Y_1} \bigr] &= &\int_0^\infty
e^{-x} \P ( Y_1 \leq x/u ) \,\mathrm{d} x
\nonumber
\\[-8pt]
\\[-8pt]
\nonumber
%& =1 - \int_0^\infty e^{-x} \P\left( Y_1 \geq x /u\right) \,\mathrm{d} x
%\nonumber\\
& =& 1 - \int_{0}^{c(u)}
e^{-x} \P ( Y_1 \geq x /u ) \,\mathrm{d} x - \int
_{c(u)}^{\infty} e^{-x} \P ( Y_1
\geq x /u ) \,\mathrm {d} x,
\end{eqnarray}
where $c(u)$ is a function depending on $u$ to be chosen. Let $c(u) = u
\log\log(u^{-1})$, then
\[
\frac{x}{u} \geq\log\log\bigl(u^{-1}\bigr)\qquad \mbox{if } x \geq
c(u); %
\]
that diverges if $u \to0^+$. It is also trivial that $c(u) =o(u^\alpha
)$ (in the case $\alpha<1$) and $c(u)= o (u \log u) $ (in the case
$\alpha= 1$) as $u \to0^+$. Hence, we can easily bound the first term
in (\ref{equa.I.0.important.part.alpha.leq.1}) by
\[
\int_{0}^{c(u)} e^{-x} \P (
Y_1 \geq x /u ) \,\mathrm{d} x \leq c(u), %
\]
that it is negligible as $u \to0^+$. We study the second term in (\ref
{equa.I.0.important.part.alpha.leq.1}), since $x/u$ diverges if $x \geq
c(u)$, we can replace $\P(Y_i \geq x /u ) $ by its asymptotic
equivalent $ u^\alpha/ x^\alpha$
\[
\int_{c(u)}^{\infty} e^{-x} \P (
Y_1 \geq x /u ) \,\mathrm{d} x \sim u^{\alpha} \int
_{c(u)}^{\infty} \frac{e^{-x}}{x^\alpha} \,\mathrm{d} x\qquad \mbox{as
} u \to0^+. %
\]
When $\alpha< 1$, we have that $\int_{c(u)}^{\infty} \frac
{e^{-x}}{x^\alpha} \,\mathrm{d} x \to\Gamma(1-\alpha) < \infty$, that
proves the statement in this case. For $\alpha= 1$, we use the
following result, that may be found in \cite{Bender1999} (Section~6.2,
Example~4):
%
%e5.2 #&#
\begin{equation}
\label{equa:develpoment.incomplete.gamma.function.zero} \int_{z}^{ \infty} \frac{ e^{- x} }{x}
\,\mathrm{d} x = -\gamma- \log z - \sum_{m \geq1}
(-1)^m \frac{z^{m}}{m(m!)} , \qquad z \to0^+,
\end{equation}
where $\gamma$ stands for the Euler--Mascheroni constant. Taking $z=
c(u)$, we obtain that
\begin{eqnarray*}
\int_{c(u)}^{\infty} \frac{e^{-x}}{x} \,\mathrm{d} x &
= &-\gamma- \log \bigl( u \log\log\bigl(u^{-1}\bigr) \bigr) - \sum
_{m \geq1} (-1)^m \frac{  (u
\log\log(u^{-1})  )^{m}}{m(m!)}
\\
& =& -\log u + o(\log u) \qquad\mbox{as } u \to0^+,
\end{eqnarray*}
finishing the proof.

\emph{We now focus on the case $p \geq2$}. We start with the following
relation:
%
%e5.3 #&#
%e5.4 #&#
\begin{eqnarray}
I_{p} (u) & =& \int_{0}^{ \infty} \bigl(p
x^{p-1} e^{-u x} -u x^{p} e^{-u x} \bigr)
\P(Y_i \geq x) \,\mathrm{d} x
\nonumber
\\
%& = \int_{0}^{ \infty} \left(p u^{-p} x^{p-1} e^{- x} - u^{-p} x^{p}
%e^{- x} \right) \P(Y_i \geq x /u ) \,\mathrm{d} x \nonumber\\
& = &\int_{0}^{c(u)} \bigl(p
u^{-p} x^{p-1} e^{- x} - u^{-p}
x^{p} e^{-
x} \bigr) \P(Y_i \geq x /u )
\,\mathrm{d} x \label
{equa.first.integral.I.p.alpha=2}
\\
&&{}+ \int_{c(u)}^{ \infty} \bigl(p u^{-p}
x^{p-1} e^{- x} - u^{-p} x^{p}
e^{- x} \bigr) \P(Y_i \geq x /u ) \,\mathrm{d} x,
\label
{equa.second.integral.I.p.alpha=2}
\end{eqnarray}
where $c(u)$ is a function depending on $u$ to be chosen. As we did
above, we will choose $c(u)$ such that it is negligible in comparison
to $u^{\alpha-p}$, but $x/u$ diverges if $x \geq c(u)$.

\emph{Suppose that $\alpha<2$ or $\alpha=2$ and $p \geq3$}. Let
$\beta\in \,]0,1[$ such that $\beta p > \alpha$ and choose $c(u)=
u^\beta$ (it is trivial that such $\beta$ does not exist if $p=\alpha
=2$). We bound (\ref{equa.first.integral.I.p.alpha=2}) by
\begin{eqnarray*}
&&\biggl| \int_{0}^{c(u)}  \bigl( p u^{-p}
x^{p-1} e^{- x} - u^{-p} x^{p}
e^{- x} \bigr) \P(Y_i \geq x /u ) \,\mathrm{d} x \biggr|
\\
&&\quad \leq u^{ p} \int_{0}^{c(u)} p
u^{-p} x^{p-1} + u^{-p} x^{p} \mathrm {d}
x
\nonumber
\\
%& = u^{ b_i} \left( u^{-b_i} c(u)^{b_i} + \frac{u^{-b_i} c(u)^{b_i+1}
%}{b_i+1} \right) \\
&&\quad = u^{(\beta+1) p} + \frac{ u^{ (\beta+1)p + 1} }{p+1} ,
\end{eqnarray*}
that is negligible in comparison to $u^{\alpha-p}$ as $u\to0^+$. We
now turn our attention to (\ref{equa.second.integral.I.p.alpha=2}),
where $x/u $ diverges as $u\to0^+$. We may replace $\P(Y_i \geq x /u )
$ by its asymptotic equivalent $ u^\alpha/ x^\alpha$, then as $u\to0^+$
%
%e5.5 #&#
\begin{eqnarray}
\label{equa.estimation.second.integral.I.p.alpha=2} &&\int_{c(u)}^{ \infty}  \bigl(p
u^{-p} x^{p-1} e^{- x} - u^{-p}
x^{p} e^{- x} \bigr) \P(Y_i \geq x /u )
\,\mathrm{d} x
\nonumber
\\
&&\quad\sim u^{\alpha-p } \int_{c(u)}^{ \infty} \bigl(p
x^{p-\alpha-1} e^{-
x} - x^{p-\alpha} e^{- x} \bigr)
\,\mathrm{d} x
\\
&&\quad = u^{\alpha-p } \alpha \Gamma(p-\alpha) - u^{\alpha-p } \int
_{0}^{c(u)} \bigl(p x^{p-\alpha-1}
e^{- x} - x^{p-\alpha} e^{- x} \bigr) \,\mathrm{d} x.\nonumber
\end{eqnarray}
Finally, the second term in the right-hand side of (\ref
{equa.estimation.second.integral.I.p.alpha=2}) is $o (u^{\alpha-p} )$
as $u\to0^+$, concluding the proof in the cases $\alpha<2$ and $\alpha
=2$, with $ p \geq2 $.

\emph{The case $p=2$ and $\alpha=2$ is obtained as above, choosing
$c(u) = u \log\log(u^{-1})$ and using the asymptotic development \textup{(\ref
{equa:develpoment.incomplete.gamma.function.zero})}. We leave the
details to the reader.} %\qed

%s6 #&#
\section{Proof of Proposition \texorpdfstring{\protect\ref
{prop:integral.representation.b's.moment.assymptotic.alpha}}{4.7}}\label
{appendix.prop:integral.representation.b's.moment.assymptotic.alpha}

In this appendix, we prove Proposition~\ref
{prop:integral.representation.b's.moment.assymptotic.alpha}. Once more,
the main idea of the proof is roughly the same for every $0<\alpha
\leq2$, but some technical adaptations are required in specific cases.
For this reason, we will present a detailed proof of the case $\alpha=
2$ and only sketch the proofs of the other cases.

Let $\kappa_N = (\log N)^2 / N$ be as in Lemma~\ref
{lem:integral.upper.tail.alpha.leq.2}. By (\ref
{equa.lem:integral.representation.b's.moment}) and Lemma~\ref
{lem:integral.upper.tail.alpha.leq.2}, we have that
\[
\E \bigl[ \eta_1^{b_1} \cdots\eta_a^{b_a}
\bigr] = \frac{1}{\Gamma(b)} \int_0^{\kappa_N}
u^{b-1} I_0(u)^{N-a} I_{b_1}(u) \cdots
I_{b_a}(u) \,\mathrm{d} u + \varepsilon_N, %
\]
where $\varepsilon_N$ decreases to zero faster than any polynomial in $N$.
Hence, it suffices to show that
%
%e6.1 #&#
\begin{equation}
\label
{important.part.integral.representation.asymptotic.N.alpha.leq.2} \lim_{N \to\infty}\frac{ N^{2 a} }{ (\log N)^{a-g}} \cdot\int
_0^{\kappa_N} u^{b-1} I_0(u)^{N-a}
I_{b_1}(u) \cdots I_{b_a}(u) \,\mathrm {d} u =
\frac{ 2^a   \prod_{i=1}^g \Gamma(b_i-2)}{ \E[Y_1]^{2 a} } \cdot\Gamma(2 a).
\end{equation}

Let $\varepsilon> 0 $, since $\lim_{N\to\infty} \kappa_N = 0$ we
apply Lemma~\ref{lem:asymptotics.I.p.neighborhood.of.zero.alpha} to
conclude that there exists a $N_0$ such that for $N$ larger than $N_0$
and $u \leq\kappa_N$
\begin{eqnarray*} (1-\varepsilon) \bigl(2 \Gamma(b_i-2)
\bigr) &\leq& I_{b_{i}} (u) / u^{2
- b_i} \leq(1+\varepsilon) \bigl(2
\Gamma(b_i-2) \bigr)\qquad \mbox{ if }  b_i \geq 3;
\\
2 (1-\varepsilon)& \leq& I_{2 }(u) / \log\bigl(u^{-1}\bigr)
\leq2 (1+\varepsilon)\qquad \mbox{if }  b_i = 2.
 \end{eqnarray*}
Since there are finitely many $b_i$'s, we may take $N_0$ such that the
inequalities hold for every $i \in\{ 1 ,2, \ldots, a \}$. As a
consequence, for $N > N_0$
%
%e6.2 #&#
\begin{eqnarray}
\label{equa.I_p.alpha=2.important.part.assymptotic} &&\int_0^{\kappa_N} u^{b-1}
I_0(u)^{N-a} I_{b_1}(u) \cdots
I_{b_a}(u) \,\mathrm{d} u
\nonumber
\\
&&\quad\geq(1-\varepsilon)^a 2^a \prod
_{i=1}^g \Gamma(b_i-2) \int
_0^{\kappa
_N} u^{b-b_1-\cdots-b_g-1+2 g} \bigl(\log
\bigl(u^{-1}\bigr) \bigr)^{a-g} I_0(u)^{N-a}
\,\mathrm{d} u
\\
&&\quad= (1-\varepsilon)^a 2^a \prod
_{i=1}^g \Gamma(b_i-2) \int
_0^{\kappa_N} u^{2 a-1} \bigl(\log
\bigl(u^{-1}\bigr) \bigr)^{a-g} I_0(u)^{N-a}
\,\mathrm{d} u,\nonumber
\end{eqnarray}
where we used $b= b_1 + \cdots+ b_a = b_1 + \cdots+ b_g + 2 (a-g) $
(a similar argument may be used to obtain a similar upper bound).
Applying Lemma~\ref{lem:asymptotics.I.p.neighborhood.of.zero.alpha} for
$I_0$, we get that
\[
\lim_{u \to0^+} \frac{I_0(u) -1}{-u \E[Y_1]} = 1. %
\]
Hence, there exists a $N_1$ such that for $N \geq N_1$ and $u \leq
\kappa_N$ (we assume that $N_1 \geq N_0$)
\[
\bigl(1 - u (1+\varepsilon) \E[Y_1] \bigr)^{N-a} \leq
I_0(u)^{N-a} \leq\bigl(1 - u (1-\varepsilon)
\E[Y_1] \bigr)^{N-a}. %
\]
Applying the above inequality in (\ref
{equa.I_p.alpha=2.important.part.assymptotic}) to obtain a lower bound,
and by the change of variables $v = u (1+\varepsilon) \E[Y_1] N$ we get
\begin{eqnarray*}
&&(1- \varepsilon)^a 2^a \prod
_{i=1}^g \Gamma(b_i-2) \int
_0^{\kappa_N} u^{2 a-1} \bigl(\log
\bigl(u^{-1}\bigr) \bigr)^{a-g} I_0(u)^{N-a}
\,\mathrm{d} u
\\
&&\quad \geq\frac{(1-\varepsilon)^a}{(1+\varepsilon)^{2 a}} \cdot\frac
{1}{N^{2 a}} \cdot\frac{ 2^a \cdot\prod_{i=1}^g \Gamma(b_i-2)}{ \E
[Y_1]^{2 a} }
\\
&&\qquad {}\times \int_0^{\gamma_N} v^{2a-1} \biggl(-
\log \biggl( \frac{ v}{N (1+\varepsilon) \E[Y_1]} \biggr) \biggr)^{a-g} \biggl(1-
\frac
{v}{N} \biggr)^{N-a} \,\mathrm{d} v,
\end{eqnarray*}
where $\gamma_N = N (1+\varepsilon) \E[Y_1] \kappa_N$, then
\[
-\log \bigl( v/ \bigl(N (1+\varepsilon) \E[Y_1] \bigr) \bigr) =
\log N \biggl( 1+ \frac{\log ((1+\varepsilon) \E[Y_1]  ) - \log v}{\log
N} \biggr), %
\]
and for $v \leq (1+\varepsilon) \E[Y_1] (\log N)^2 =\gamma_N$
%
%e6.3 #&#
\begin{equation}
\label{kappa.N.objective} \frac{ | \log ((1+\varepsilon) \E[Y_1] ) - \log v
|}{\log N} \to0 \qquad\mbox{as } N \to\infty.
\end{equation}
Moreover, (\ref{kappa.N.objective}) decays uniformly to zero for $v
\leq\gamma_N$. \emph{We bring to the reader's attention the choice of
$\kappa_N$ in Lemma~\textup{\ref{lem:integral.upper.tail.alpha.leq.2}}, because
it was chosen such that \textup{(\ref{kappa.N.objective})} decays to zero
uniformly}. Then there exists a $N_2$ such that for $N \geq N_2$ (we
assume that $N_2 \geq N_1$)
\[
(1-\varepsilon) \log N \leq-\log \bigl( v/ \bigl(N (1+\varepsilon) \E
[Y_1] \bigr) \bigr) \leq (1+\varepsilon) \log N\qquad \mbox{for every }
v \leq\gamma_N. %
\]
Then, for $N \geq N_2$ we may further bound (\ref
{equa.I_p.alpha=2.important.part.assymptotic}) and obtain
%
%e6.4 #&#
\begin{eqnarray}
\label{equa.final.lower.bound} &&\int_0^{\kappa_N} u^{b-1}
I_0(u)^{N-a} I_{b_1}(u) \cdots
I_{b_a}(u) \,\mathrm{d} u
\nonumber
\\[-8pt]
\\[-8pt]
\nonumber
&&\quad\geq\frac{(1-\varepsilon)^{2a-g}}{(1+\varepsilon)^{2 a}} \cdot\frac
{(\log N)^{a-g}}{N^{2 a}} \cdot \frac{ 2^a   \prod_{i=1}^g \Gamma
(b_i-2)}{ \E[Y_1]^{2 a} }
\cdot\int_0^{\gamma_N} v^{2 a-1} \biggl(1-
\frac{v}{N} \biggr)^{N-a} \,\mathrm{d} v.
\end{eqnarray}
Since $v \leq\gamma_N$, both $v/N$ and $v^2/N$ decay to zero as $N\to
\infty$. We also have that
\[
\biggl(1-\frac{v}{N} \biggr)^{N-a} = \exp \bigl( -v +
\mathcal{O} \bigl(v^2/N \bigr) \bigr)\qquad \mbox{as } N \to\infty.
\]
As a consequence, the following limit holds:
\[
\lim_{N\to\infty} \int_0^{\gamma_N}
v^{2 a-1} \biggl(1-\frac{v}{N} \biggr)^{N-a} \,\mathrm{d}
v = \Gamma(2 a). %
\]
Since $\varepsilon$ in (\ref{equa.final.lower.bound}) is arbitrary, we
have that
\[
\liminf_{N \to\infty} \E \bigl[ \eta_1^{b_1}
\cdots\eta_a^{b_a} \bigr] \cdot\frac{ N^{2 a} }{ (\log N)^{a-g}} \geq
\frac{ 2^a   \prod_{i=1}^g \Gamma(b_i-2)}{ \E[Y_1]^{2 a} } \cdot\Gamma(2 a). %
\]
We obtain an upper bound for the $\limsup$ using a similar argument
with the obvious changes, and we leave the details to the reader.
Hence, the limit in (\ref
{important.part.integral.representation.asymptotic.N.alpha.leq.2})
holds, which proves the statement.

We now sketch the proof of Proposition~\ref
{prop:integral.representation.b's.moment.assymptotic.alpha} in the
remaining cases ($\alpha<2$), and we explain briefly how to overcome
possible difficulties. \emph{The case $1<\alpha<2$ has no further
difficulties and we leave the details of the proof to the reader. In
the case $\alpha=1$, the relevant term to estimate is of the form}:
\[
\Gamma(b_1-1) \cdots\Gamma(b_a-1) \cdot\int
_0^{\kappa_N} u^{b-1} I_0(u)^{N-a}
u^{1-b_i} \cdots u^{1-b_a} \,\mathrm{d} u. %
\]
By Lemma~\ref{lem:asymptotics.I.p.neighborhood.of.zero.alpha},
$I_0(u)^{N-a} \cong (1+u \log u)^{N-a}$. Then, by the change of
variables $v = u N \log N$, we obtain an expression of the form:
\[
\frac{\prod\Gamma(b_i-1)}{ (N\log N)^a } \cdot \int_0^{\kappa_N   N
\log N }
v^{a-1} \biggl( 1 + \frac{v}{N \log N} \log\frac{v}{N \log N}
\biggr)^{N-a} \,\mathrm{d} v. %
\]
Since $v \leq\kappa_N   N \log N = (\log N )^3$, the equation inside
of the parentheses has the following asymptotic behavior as $N\to\infty$:
\begin{eqnarray*}
1 + \frac{v}{N \log N} \log \biggl( \frac{v}{N \log N} \biggr) & =& 1 -
\frac{v}{N} \cdot \biggl(1 + \frac{\log\log N - \log v}{\log N} \biggr)
\\
& \cong& 1 - \frac{v}{N},
\end{eqnarray*}
then we may proceed as in the case $\alpha=2$ to prove the statement.
\emph{In the case $\alpha<1$, we will arrive to an equation of the form}
\[
\prod\alpha\Gamma(b_i-\alpha) \int_0^{\kappa_N}
u^{a \alpha-1} I_0(u)^{N-a} \,\mathrm{d} u. %
\]
We then use the development of $I_0(u)$ in a neighborhood of zero and
the change of variables $v = u^{\alpha} \Gamma(1-\alpha) N$ to obtain
\[
\frac{\prod\alpha\Gamma(b_i-\alpha)}{\alpha\Gamma(1-\alpha)^a N^a} \int_0^{\kappa_N^{\alpha} \Gamma(1-\alpha) N }
v^{a-1} \biggl( 1-\frac
{v}{N} \biggr)^{N-a} \,\mathrm{d}
v, %
\]
that finishes the proof. %\qed
\end{appendix}

\section*{Acknowledgments}
I would like to thank my supervisor, Professor Francis Comets, for
suggesting this problem, for the helpful discussions and for his
patient guidance.

%\bibliographystyle{plain}
%%\bibliography{coalescent}

\begin{thebibliography}{20}
% pybtex-1.35. Style name=bej, version=1.42, label_style=nolabel,
%sorting_style=complex, cfg=None, language=None.

%b1 ###
%b1 #&#
\bibitem{Bender1999}
%
\begin{bbook}[mr]
\bauthor{\bsnm{Bender},~\bfnm{Carl~M.}\binits{C.M.}} \AND
\bauthor{\bsnm{Orszag},~\bfnm{Steven~A.}\binits{S.A.}}
(\byear{1999}).
\btitle{Advanced Mathematical Methods for Scientists and Engineers. I:
Asymptotic Methods and Perturbation Theory}.
\blocation{New York}:
\bpublisher{Springer}.
\bid{doi={10.1007/978-1-4757-3069-2}, mr={1721985}}
\end{bbook}
%

\bptok{imsref}%
% NOT OUTPUTTED:
% doi = http://dx.doi.org/10.1007/978-1-4757-3069-2
% isbn = 0-387-98931-5
% fpage = xiv+593
\endbibitem

%b2 ###
%b2 #&#
\bibitem{Berestycki2013}
%
\begin{barticle}[mr]
\bauthor{\bsnm{Berestycki},~\bfnm{Julien}\binits{J.}},
\bauthor{\bsnm{Berestycki},~\bfnm{Nathana{\"e}l}\binits{N.}} \AND
\bauthor{\bsnm{Schweinsberg},~\bfnm{Jason}\binits{J.}}
(\byear{2013}).
\btitle{The genealogy of branching {B}rownian motion with absorption}.
\bjournal{Ann. Probab.}
\bvolume{41}
\bpages{527--618}.
\bid{doi={10.1214/11-AOP728}, issn={0091-1798}, mr={3077519}}
\end{barticle}
%

\bptok{imsref}%
% NOT OUTPUTTED:
% number = 2
% doi = http://dx.doi.org/10.1214/11-AOP728
% fjournal = The Annals of Probability
\endbibitem

%b3 ###
%b3 #&#
\bibitem{Bolthausen1998}
%
\begin{barticle}[mr]
\bauthor{\bsnm{Bolthausen},~\bfnm{E.}\binits{E.}} \AND
\bauthor{\bsnm{Sznitman},~\bfnm{A.-S.}\binits{A.-S.}}
(\byear{1998}).
\btitle{On {R}uelle's probability cascades and an abstract cavity method}.
\bjournal{Comm. Math. Phys.}
\bvolume{197}
\bpages{247--276}.
\bid{doi={10.1007/s002200050450}, issn={0010-3616}, mr={1652734}}
\end{barticle}
%

\bptok{imsref}%
% NOT OUTPUTTED:
% number = 2
% doi = http://dx.doi.org/10.1007/s002200050450
% coden = CMPHAY
% fjournal = Communications in Mathematical Physics
\endbibitem

%b4 ###
%b4 #&#
\bibitem{Brunet2004}
%
\begin{barticle}[mr]
\bauthor{\bsnm{Brunet},~\bfnm{{\'E}ric}\binits{{\'E}.}} \AND
\bauthor{\bsnm{Derrida},~\bfnm{Bernard}\binits{B.}}
(\byear{2004}).
\btitle{Exactly soluble noisy traveling-wave equation appearing in the
problem of directed polymers in a random medium}.
\bjournal{Phys. Rev. E (3)}
\bvolume{70}
\bpages{016106, 5}.
\bid{doi={10.1103/PhysRevE.70.016106}, issn={1539-3755}, mr={2125704}}
\bptnote{check pages}%
\end{barticle}
%

\bptok{imsref}%
% NOT OUTPUTTED:
% number = 1
% doi = http://dx.doi.org/10.1103/PhysRevE.70.016106
% coden = PLEEE8
% fjournal = Physical Review E. Statistical, Nonlinear, and Soft Matter
%Physics
\endbibitem

%b5 ###
%b5 #&#
\bibitem{Brunet2013}
%
\begin{barticle}[mr]
\bauthor{\bsnm{Brunet},~\bfnm{{\'E}ric}\binits{{\'E}.}} \AND
\bauthor{\bsnm{Derrida},~\bfnm{Bernard}\binits{B.}}
(\byear{2013}).
\btitle{Genealogies in simple models of evolution}.
\bjournal{J. Stat. Mech. Theory Exp.}
\bvolume{1}
\bpages{P01006, 20}.
\bid{issn={1742-5468}, mr={3036206}}
\bptnote{check pages}%
\end{barticle}
%

\bptok{imsref}%
% NOT OUTPUTTED:
% fjournal = Journal of Statistical Mechanics: Theory and Experiment
\endbibitem

%b6 ###
%b6 #&#
\bibitem{Brunet2008}
%
\begin{barticle}[auto:parserefs-M02]
\bauthor{\bsnm{{B}runet},~\bfnm{E.}\binits{E.}},
\bauthor{\bsnm{{D}errida},~\bfnm{B.}\binits{B.}} \AND
\bauthor{\bsnm{{D}amien},~\bfnm{S.}\binits{S.}}
(\byear{2008}).
\btitle{Universal tree structures in directed polymers and models of
evolving populations}.
\bjournal{Phys. Rev. E}
\bvolume{78}
\bpages{061102}.
\end{barticle}
%

\bptok{imsref}%
\endbibitem

%b7 ###
%b7 #&#
\bibitem{Brunet2006}
%
\begin{barticle}[mr]
\bauthor{\bsnm{Brunet},~\bfnm{E.}\binits{E.}},
\bauthor{\bsnm{Derrida},~\bfnm{B.}\binits{B.}},
\bauthor{\bsnm{Mueller},~\bfnm{A.~H.}\binits{A.H.}} \AND
\bauthor{\bsnm{Munier},~\bfnm{S.}\binits{S.}}
(\byear{2006}).
\btitle{Noisy traveling waves: Effect of selection on genealogies}.
\bjournal{Europhys. Lett.}
\bvolume{76}
\bpages{1--7}.
\bid{doi={10.1209/epl/i2006-10224-4}, issn={0295-5075}, mr={2299937}}
\end{barticle}
%

\bptok{imsref}%
% NOT OUTPUTTED:
% number = 1
% doi = http://dx.doi.org/10.1209/epl/i2006-10224-4
% fjournal = Europhysics Letters
\endbibitem

%b8 ###
%b8 #&#
\bibitem{Brunet2007}
%
\begin{barticle}[mr]
\bauthor{\bsnm{Brunet},~\bfnm{{\'E}.}\binits{{\'E}.}},
\bauthor{\bsnm{Derrida},~\bfnm{B.}\binits{B.}},
\bauthor{\bsnm{Mueller},~\bfnm{A.~H.}\binits{A.H.}} \AND
\bauthor{\bsnm{Munier},~\bfnm{S.}\binits{S.}}
(\byear{2007}).
\btitle{Effect of selection on ancestry: An exactly soluble case and
its phenomenological generalization}.
\bjournal{Phys. Rev. E (3)}
\bvolume{76}
\bpages{041104, 20}.
\bid{doi={10.1103/PhysRevE.76.041104}, issn={1539-3755}, mr={2365627}}
\end{barticle}
%

\bptok{imsref}%
% NOT OUTPUTTED:
% number = 4
% doi = http://dx.doi.org/10.1103/PhysRevE.76.041104
% coden = PLEEE8
% fjournal = Physical Review E. Statistical, Nonlinear, and Soft Matter
%Physics
\endbibitem

%b9 ###
%b9 #&#
\bibitem{Comets2013}
%
\begin{barticle}[mr]
\bauthor{\bsnm{Comets},~\bfnm{Francis}\binits{F.}},
\bauthor{\bsnm{Quastel},~\bfnm{Jeremy}\binits{J.}} \AND
\bauthor{\bsnm{Ram{\'{\i}}rez},~\bfnm{Alejandro~F.}\binits{A.F.}}
(\byear{2013}).
\btitle{Last passage percolation and traveling fronts}.
\bjournal{J. Stat. Phys.}
\bvolume{152}
\bpages{419--451}.
\bid{doi={10.1007/s10955-013-0779-8}, issn={0022-4715}, mr={3082639}}
\end{barticle}
%

\bptok{imsref}%
% NOT OUTPUTTED:
% number = 3
% doi = http://dx.doi.org/10.1007/s10955-013-0779-8
% fjournal = Journal of Statistical Physics
\endbibitem

%b10 ###
%b10 #&#
\bibitem{Cook1990}
%
\begin{barticle}[mr]
\bauthor{\bsnm{Cook},~\bfnm{J.}\binits{J.}} \AND
\bauthor{\bsnm{Derrida},~\bfnm{B.}\binits{B.}}
(\byear{1990}).
\btitle{Directed polymers in a random medium: {$1/d$} expansion and the
{$n$}-tree approximation}.
\bjournal{J. Phys. A}
\bvolume{23}
\bpages{1523--1554}.
\bid{issn={0305-4470}, mr={1048783}}
\end{barticle}
%

\bptok{imsref}%
% NOT OUTPUTTED:
% url = http://stacks.iop.org/0305-4470/23/1523
% number = 9
% coden = JPHAC5
% fjournal = Journal of Physics. A. Mathematical and General
\endbibitem

%b11 ###
%b11 #&#
\bibitem{Cortines2013}
%
\begin{barticle}[mr]
\bauthor{\bsnm{Cortines},~\bfnm{Aser}\binits{A.}}
(\byear{2014}).
\btitle{Front velocity and directed polymers in random medium}.
\bjournal{Stochastic Process. Appl.}
\bvolume{124}
\bpages{3698--3723}.
\bid{doi={10.1016/j.spa.2014.05.012}, issn={0304-4149}, mr={3249352}}
\end{barticle}
%

\bptok{imsref}%
% NOT OUTPUTTED:
% number = 11
% doi = http://dx.doi.org/10.1016/j.spa.2014.05.012
% fjournal = Stochastic Processes and their Applications
\endbibitem

%b12 ###
%b12 #&#
\bibitem{Huillet2011}
%
\begin{barticle}[mr]
\bauthor{\bsnm{Huillet},~\bfnm{Thierry}\binits{T.}} \AND
\bauthor{\bsnm{M{\"o}hle},~\bfnm{Martin}\binits{M.}}
(\byear{2011}).
\btitle{Population genetics models with skewed fertilities: A forward
and backward analysis}.
\bjournal{Stoch. Models}
\bvolume{27}
\bpages{521--554}.
\bid{doi={10.1080/15326349.2011.593411}, issn={1532-6349}, mr={2827443}}
\end{barticle}
%

\bptok{imsref}%
% NOT OUTPUTTED:
% number = 3
% doi = http://dx.doi.org/10.1080/15326349.2011.593411
% coden = CSSME8
% fjournal = Stochastic Models
\endbibitem

%b13 ###
%b13 #&#
\bibitem{Huillet2013}
%
\begin{barticle}[auto:parserefs-M02]
\bauthor{\bsnm{Huillet},~\bfnm{T.}\binits{T.}} \AND
\bauthor{\bsnm{M{\"{o}}hle},~\bfnm{M.}\binits{M.}}
(\byear{2013}).
\btitle{On the extended {M}oran model and its relation to coalescents
with multiple collisions}.
\bjournal{Theoretical Population Biology}
\bvolume{87}
\bpages{5--14}.
\end{barticle}
%

\bptok{imsref}%
\endbibitem

%b14 ###
%b14 #&#
\bibitem{Kingman1982}
%
\begin{barticle}[mr]
\bauthor{\bsnm{Kingman},~\bfnm{J.~F.~C.}\binits{J.F.C.}}
(\byear{1982}).
\btitle{On the genealogy of large populations}.
\bjournal{J. Appl. Probab.}
\bvolume{19A}
\bpages{27--43}.
%\bnote{Essays in statistical science}.
\bid{issn={0170-9739}, mr={0633178}}
\end{barticle}
%

\bptok{imsref}%
% NOT OUTPUTTED:
% fjournal = Journal of Applied Probability
\endbibitem

%b15 ###
%b15 #&#
\bibitem{Mohle1999}
%
\begin{barticle}[mr]
\bauthor{\bsnm{M{\"o}hle},~\bfnm{M.}\binits{M.}}
(\byear{1999}).
\btitle{Weak convergence to the coalescent in neutral population models}.
\bjournal{J. Appl. Probab.}
\bvolume{36}
\bpages{446--460}.
\bid{issn={0021-9002}, mr={1724816}}
\end{barticle}
%

\bptok{imsref}%
% NOT OUTPUTTED:
% number = 2
% coden = JPRBAM
% fjournal = Journal of Applied Probability
\endbibitem

%b16 ###
%b16 #&#
\bibitem{Mohle2000}
%
\begin{barticle}[mr]
\bauthor{\bsnm{M{\"o}hle},~\bfnm{M.}\binits{M.}}
(\byear{2000}).
\btitle{Total variation distances and rates of convergence for
ancestral coalescent processes in exchangeable population models}.
\bjournal{Adv. in Appl. Probab.}
\bvolume{32}
\bpages{983--993}.
\bid{doi={10.1239/aap/1013540343}, issn={0001-8678}, mr={1808909}}
\end{barticle}
%

\bptok{imsref}%
% NOT OUTPUTTED:
% number = 4
% doi = http://dx.doi.org/10.1239/aap/1013540343
% coden = AAPBBD
% fjournal = Advances in Applied Probability
\endbibitem

%b17 ###
%b17 #&#
\bibitem{Mohle2001}
%
\begin{barticle}[mr]
\bauthor{\bsnm{M{\"o}hle},~\bfnm{Martin}\binits{M.}} \AND
\bauthor{\bsnm{Sagitov},~\bfnm{Serik}\binits{S.}}
(\byear{2001}).
\btitle{A classification of coalescent processes for haploid
exchangeable population models}.
\bjournal{Ann. Probab.}
\bvolume{29}
\bpages{1547--1562}.
\bid{doi={10.1214/aop/1015345761}, issn={0091-1798}, mr={1880231}}
\end{barticle}
%

\bptok{imsref}%
% NOT OUTPUTTED:
% number = 4
% doi = http://dx.doi.org/10.1214/aop/1015345761
% coden = APBYAE
% fjournal = The Annals of Probability
\endbibitem

%b18 ###
%b18 #&#
\bibitem{Pitman1999}
%
\begin{barticle}[mr]
\bauthor{\bsnm{Pitman},~\bfnm{Jim}\binits{J.}}
(\byear{1999}).
\btitle{Coalescents with multiple collisions}.
\bjournal{Ann. Probab.}
\bvolume{27}
\bpages{1870--1902}.
\bid{doi={10.1214/aop/1022677552}, issn={0091-1798}, mr={1742892}}
\end{barticle}
%

\bptok{imsref}%
% NOT OUTPUTTED:
% number = 4
% doi = http://dx.doi.org/10.1214/aop/1022677552
% coden = APBYAE
% fjournal = The Annals of Probability
\endbibitem

%b19 ###
%b19 #&#
\bibitem{Schweinsberg2000}
%
\begin{barticle}[mr]
\bauthor{\bsnm{Schweinsberg},~\bfnm{Jason}\binits{J.}}
(\byear{2000}).
\btitle{Coalescents with simultaneous multiple collisions}.
\bjournal{Electron. J. Probab.}
\bvolume{5}
\bpages{50 pp. (electronic)}.
\bid{doi={10.1214/EJP.v5-68}, issn={1083-6489}, mr={1781024}}
\bptnote{check pages}%
\end{barticle}
%

\bptok{imsref}%
% NOT OUTPUTTED:
% doi = http://dx.doi.org/10.1214/EJP.v5-68
% fjournal = Electronic Journal of Probability
\endbibitem

%b20 ###
%b20 #&#
\bibitem{Schweinsberg2003}
%
\begin{barticle}[mr]
\bauthor{\bsnm{Schweinsberg},~\bfnm{Jason}\binits{J.}}
(\byear{2003}).
\btitle{Coalescent processes obtained from supercritical
{G}alton--{W}atson processes}.
\bjournal{Stochastic Process. Appl.}
\bvolume{106}
\bpages{107--139}.
\bid{issn={0304-4149}, mr={1983046}}
\end{barticle}
%

\bptok{imsref}%
% NOT OUTPUTTED:
% url =
%http://www.sciencedirect.com/science?_ob=GatewayURL&_origin=MR&_method=citationSearch&_piikey=s0304414903000280&_version=1&md5=438c179f2400fbe4f5f43d781809d0f6
% number = 1
% coden = STOPB7
% fjournal = Stochastic Processes and their Applications
\endbibitem
\end{thebibliography}

% imsref loaded by akundreckaite, 2015-06-23 10:32:48
%

%\begin{appendix}
%\section{}
%\end{appendix}

% zodis "Acknowledgments" paliekamas pagal autoriu
%\section*{Acknowledgements}

%\begin{supplement}%[id=suppA]
%\sname{Supplement A}
%\stitle{}
%\slink[doi]{10.3150/00-BEJXXXXSUPP} %[doi,text={...}] - jei reikia
%suskaldyti doi
%\sdatatype{.pdf}
%\sfilename{BEJ000\_supp.pdf}
%\sdescription{}
%\end{supplement}

%\begin{thebibliography}{00}
%\bibitem[\protect\citeauthoryear{}{()}]{r1}
%\bibitem{r1}
%\end{thebibliography}

\printhistory
\end{document}